\documentclass{article}

\usepackage[english]{babel}
\usepackage[utf8]{inputenc}
\usepackage{bbm}
\usepackage{amsmath}
\usepackage{amsfonts}
\usepackage{graphicx}
\usepackage{multirow}
\usepackage[colorinlistoftodos]{todonotes}
\usepackage{blindtext}
\usepackage{amssymb}
\usepackage{algorithm}
\usepackage{amsthm}
\numberwithin{equation}{section}
\makeatletter
\newcommand*{\rom}[1]{\expandafter\@slowromancap\romannumeral #1@}

\newtheorem{theorem}{Theorem}[section]

\newtheorem{lemma}[theorem]{Lemma}
\newtheorem{proposition}{Proposition}[section]

\theoremstyle{remark}
\newtheorem*{remark}{Remark}

\theoremstyle{definition}

\title{Mallows permutation models with $L^1$ and $L^2$ distances \rom{1}: hit and run algorithms and mixing times}

\author{Chenyang Zhong\thanks{Department of Statistics, Stanford University}}

\date{\today}

\begin{document}
\maketitle
\begin{abstract}
Mallows permutation model, introduced by Mallows in statistical ranking theory, is a class of non-uniform probability measures on the symmetric group $S_n$. The model depends on a distance metric $d(\sigma,\tau)$ on $S_n$, which can be chosen from a host of metrics on permutations. In this paper, we focus on Mallows permutation models with $L^1$ and $L^2$ distances, respectively known in the statistics literature as Spearman's footrule and Spearman's rank correlation.

Unlike most of the random permutation models that have been analyzed in the literature, Mallows permutation models with $L^1$ and $L^2$ distances do not have an explicit expression for their normalizing constants. This poses challenges to the task of \emph{sampling} from these Mallows models. In this paper, we consider \emph{hit and run algorithms} for sampling from both models. Hit and run algorithms are a unifying class of Markov chain Monte Carlo (MCMC) algorithms including the celebrated Swendsen-Wang and data augmentation algorithms. For both models, we show order $\log{n}$ mixing time upper bounds for the hit and run algorithms. This demonstrates much faster mixing of the hit and run algorithms compared to local MCMC algorithms such as the Metropolis algorithm. The proof of the results on mixing times is based on the path coupling technique, for which a novel coupling for permutations with one-sided restrictions is involved. 

Extensions of the hit and run algorithms to weighted versions of the above models, a two-parameter permutation model that involves the $L^1$ distance and Cayley distance, and lattice permutation models in dimensions greater than or equal to $2$ are also discussed. The order $\log{n}$ mixing time upper bound pertains to the two-parameter permutation model. 
\end{abstract}

\section{Introduction}\label{Sect.1}

Random permutations are ubiquitous in various fields including combinatorics, probability, and statistics. Mallows permutation model, introduced by Mallows \cite{Mal} in statistical ranking theory, is a class of non-uniform probability measures on the symmetric group $S_n$. The model depends on a distance metric $d(\sigma,\tau)$ on $S_n$, a location parameter $\sigma_0\in S_n$, and a scale parameter $\beta$. The distance metric $d(\sigma,\tau)$ can be chosen from a host of metrics on permutations, which we will discuss in detail in Section \ref{Sect.1.1}. Under the model, the probability of generating a permutation $\sigma\in S_n$ is given by
\begin{equation*}
    \mathbb{P}_{d,\sigma_0,\beta}(\sigma)\propto \exp(-\beta d(\sigma,\sigma_0)).
\end{equation*}
For $\beta>0$ and a reasonable choice of the distance metric $d(\sigma,\tau)$, the model is biased towards the permutation $\sigma_0$. The reader is referred to \cite{BDJ2,Cri,DR,Mar} and Section \ref{Sect.1.1} below for further discussions on Mallows permutation model.

In this paper, we focus on Mallows permutation models with $L^1$ or $L^2$ distance. The $L^1$ and $L^2$ distances are defined respectively as
\begin{equation}\label{Dis}
    H(\sigma,\tau):=\sum_{i=1}^n|\sigma(i)-\tau(i)|, \quad  \tilde{H}(\sigma,\tau):=\sum_{i=1}^n (\sigma(i)-\tau(i))^2,
\end{equation}
for any $\sigma,\tau\in S_n$. These distances are also known as Spearman's footrule and Spearman's rank correlation in the statistics literature. In the rest of the paper, we also refer to these two Mallows models as the $L^1\slash L^2$ model. We restrict our attention to the case when the location parameter $\sigma_0=Id$, the identity permutation, and the scale parameter $\beta$ is positive. We denote by $\mathbb{P}_{\beta}$ and $\tilde{\mathbb{P}}_{\beta}$ the probability measures that correspond to the $L^1$ and $L^2$ models under the above scenarios: for any $\sigma\in S_n$,
\begin{equation*}
    \mathbb{P}_{\beta}(\sigma)=Z_{\beta}^{-1}\exp(-\beta H(\sigma,Id)),\quad \tilde{\mathbb{P}}_{\beta}(\sigma)=\tilde{Z}_{\beta}^{-1}\exp(-\beta \tilde{H}(\sigma,Id)),
\end{equation*}
where $Z_{\beta},\tilde{Z}_{\beta}$ are the normalizing constants. We also denote by $\mathbb{E}_{\beta}$ and $\tilde{\mathbb{E}}_{\beta}$ the corresponding expectation operators. 

A typical question from probabilistic combinatorics is the following \cite{BDJ2}: Picking a random permutation from the $L^1$ or $L^2$ model, what does it ``look like''? This may involve various features of the permutation, such as the length of the longest increasing subsequence or the number of fixed points. To explore these features visually, one can try to generate a large number of samples from either of the models and make histograms of the corresponding features. In this paper, we consider \emph{hit and run algorithms} for sampling from both the $L^1$ and $L^2$ models. Hit and run algorithms are a unifying class of Markov chain Monte Carlo (MCMC) algorithms including the celebrated Swendsen-Wang algorithm for sampling from the Ising model and the data augmentation algorithm that is widely used in statistics \cite{AD}. A brief overview of this class of algorithms will be given in Section \ref{Sect.1.2} below.

We present here some visualization regarding both the $L^1$ and $L^2$ models. Under the uniform distribution on $S_n$, the number of fixed points approximately follows the Poisson distribution with parameter $1$ (see Subfigure (a1) of Figure \ref{figa} for a comparison between the number of fixed points and Poisson distribution with parameter $1$ when $n=1000$). How does this change under the $L^1$ and $L^2$ models? Figure \ref{figb} shows histograms of the number of fixed points when $n=1000$ for various choices of the scale parameter $\beta$ for both the $L^1$ and $L^2$ models. We see that for $\beta=c\slash n$ (or $\beta=c\slash n^2$, respectively) with $c\in\{0.5,1\}$, under the $L^1$ model (or the $L^2$ model, respectively), the distribution of the number of fixed points is approximately Poisson with parameter depending on $\beta$ and greater than $1$. Some scatter plots (plots of the points $\{(i,\sigma(i))\}_{i=1}^n$ in the plane for $\sigma$ drawn from a particular model) of the $L^1$ and $L^2$ models are also shown in Figure \ref{figc}. 

\begin{figure}[H]
  \centering
  \includegraphics[
  height=\textwidth,
  width=\textwidth
  ]{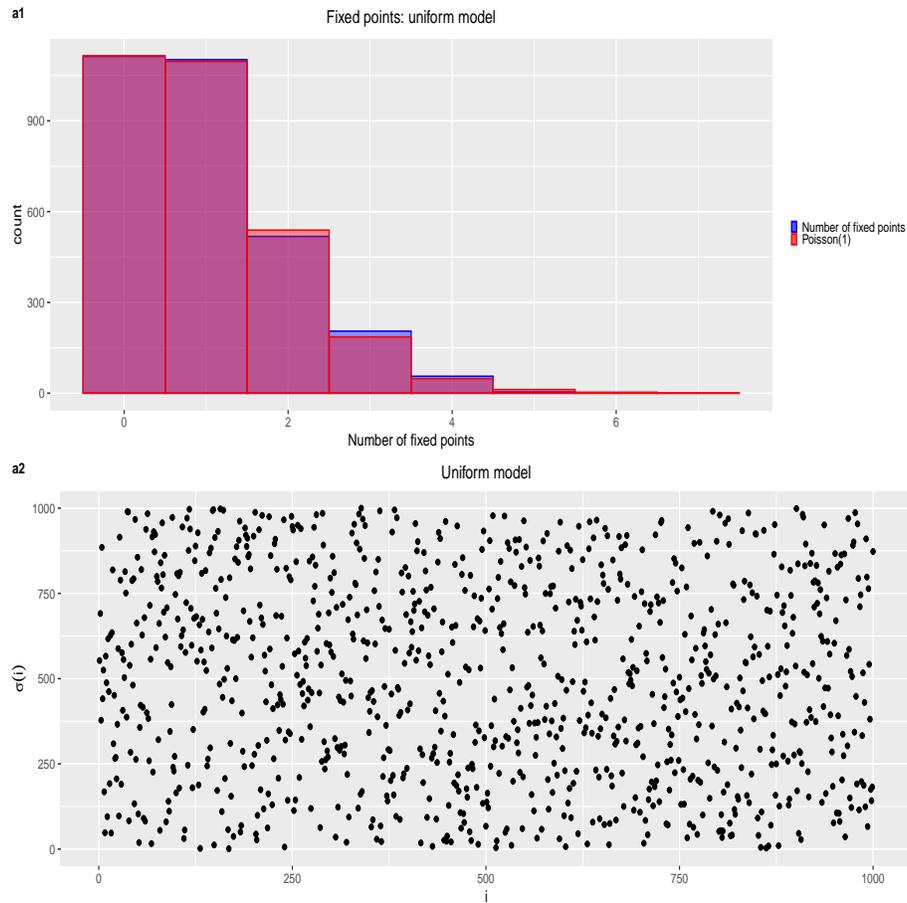}
 \caption{Histogram of the number of fixed points (a1) and scatter plot (a2) of uniform random permutation with $n=1000$. For (a1), the histogram in blue is based on the number of fixed points of $3000$ independent uniform random permutations; for comparison, we also present the histogram of $3000$ independent Poisson random variables with parameter $1$ in red.}
 \label{figa}
\end{figure} 

\begin{figure}[H]
  \centering
  \includegraphics[
  height=\textwidth,
  width=\textwidth
  ]{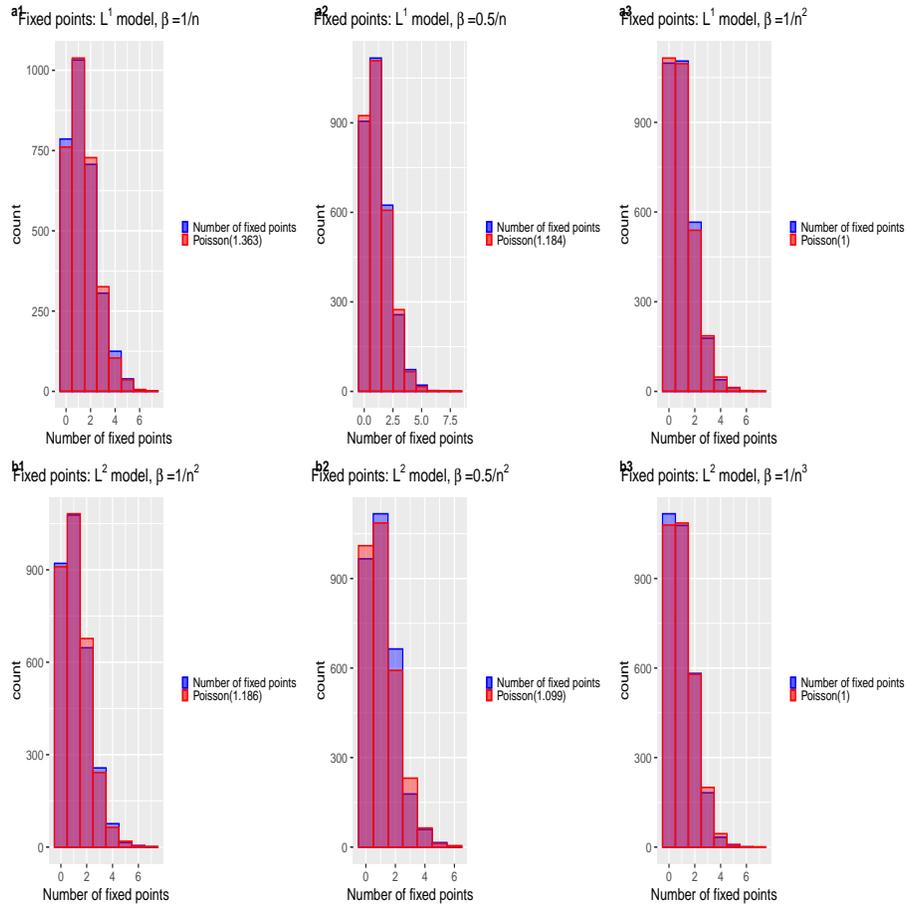}
 \caption{Histograms of the number of fixed points of the $L^1$ and $L^2$ models with $n=1000$ and various choices of the parameter. The histograms are based on the first $4000$ samples (dropping the first $1000$ samples as burn-in) of the corresponding hit and run algorithm (introduced in Section \ref{Sect.2}) started at the identity permutation. We also present the histograms of $3000$ independent Poisson random variables (with varying parameters) in red for comparison.}
 \label{figb}
\end{figure} 

\begin{figure}[H]
  \centering
  \includegraphics[
  height=\textwidth,
  width=\textwidth
  ]{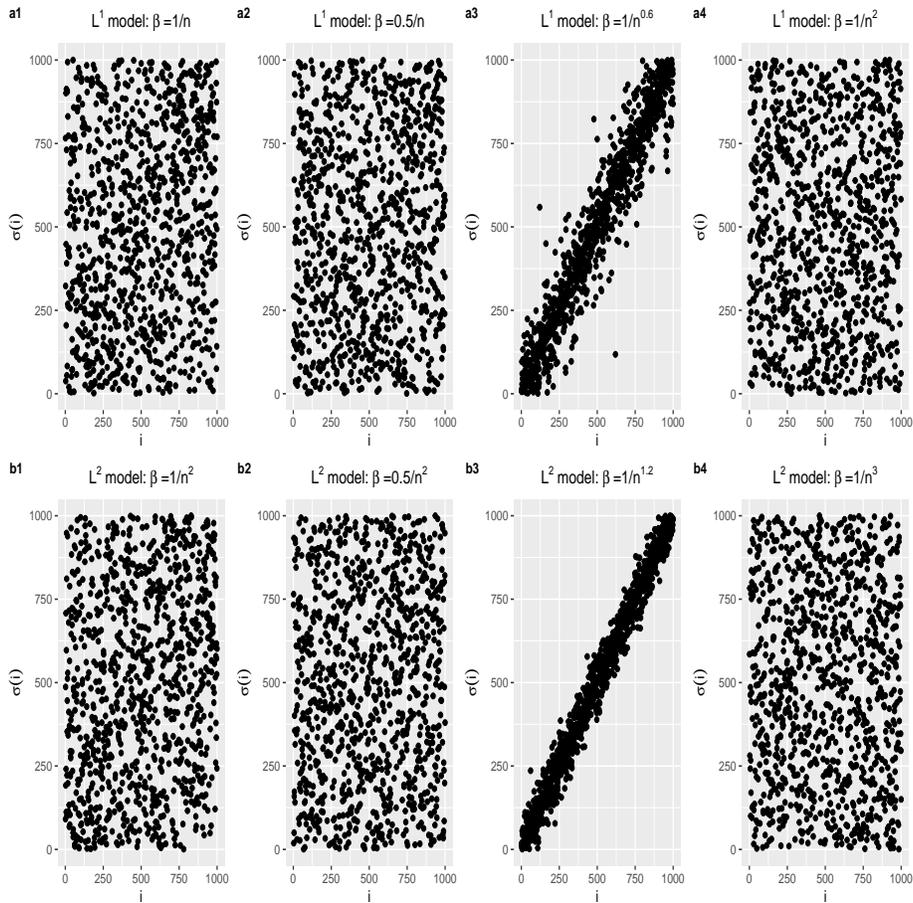}
 \caption{Scatter plots of the $L^1$ and $L^2$ models with $n=1000$ and various choices of the parameter. The scatter plots are based on the sample collected at the $4000$th step of the corresponding hit and run algorithm (introduced in Section \ref{Sect.2}) started at the identity permutation.}
 \label{figc}
\end{figure}

The scale parameter $\beta>0$ in Mallows permutation model has quite delicate behavior. If $\beta$ is too close to zero, the model is indistinguishable from the uniform distribution on $S_n$. If $\beta$ is large, the model exhibits a band structure; that is, if $\sigma\in S_n$ is a random permutation drawn from the model, the points $\{(i,\sigma(i))\}_{i=1}^n$ are concentrated in a band along the diagonal $\{(i,i):i\in [n]\}$ with width of order smaller than $n$. For the $L^1$ model, the critical magnitude of $\beta$ turns out to be $\beta\sim c\slash n$ with $c>0$ fixed. For the $L^2$ model, the critical magnitude turns out to be $\beta \sim c\slash n^2$ with $c>0$ fixed. For either of the models, when $\beta$ is chosen to be of critical magnitude and $\sigma\in S_n$ is drawn from the model, the random probability measure $\nu_{\sigma}=n^{-1}\sum_{i=1}^n \delta_{(i\slash n,\sigma(i)\slash n)}$ on $[0,1]^2$ converges weakly in probability. Mathematical backup for these claims is in \cite{FM,M1,M2,Zho2}. It also shows strikingly in simulations (see Figures \ref{figa}-\ref{figc} above). 

For both the $L^1$ and $L^2$ models, we show order $\log{n}$ mixing time upper bounds for the hit and run algorithms for certain ranges of the parameter $\beta$ that are of practical interest (they correspond to the critical magnitude as explained in the previous paragraph). The concrete results are in Section \ref{Sect.3}. The mixing time results demonstrate that the hit and run algorithms mix much faster than local MCMC algorithms such as the Metropolis algorithm. The proof of these results is based on the path coupling technique \cite{BD,LP}, a useful technique for bounding rates of convergence of finite Markov chains. To construct the path coupling, a novel coupling for permutations with one-sided restrictions is involved (see Section \ref{Sect.3} below for details). 

Following a suggestion by Paul Switzer, we further extend the hit and run algorithms to the analogs of the $L^1$ and $L^2$ models with weighted distances in Section \ref{Sect.2.3}. We also introduce generalizations of the hit and run algorithms for sampling from a wider class of permutation models, such as a two-parameter permutation model that involves the $L^1$ distance and Cayley distance (see Section \ref{Sect.1.1} for the definition of Cayley distance) in Section \ref{Sect.4}, and higher dimensional analogs of the $L^1$ and $L^2$ models that are of interest in the mathematical physics community in Section \ref{Sect.5}. 

The algorithms and analysis presented here can be used beyond simulation. They give a representation of these Mallows models that is useful in proving limit theorems about the distribution of ``features'' such as the number of fixed points, cycles, band structure, and the length of the longest increasing subsequence. See \cite{Zho2,Zho3,Zho4}.

In the rest of this Introduction, we review background material on Mallows permutation model and hit and run algorithms in Sections \ref{Sect.1.1}-\ref{Sect.1.2}, respectively.


\subsection{Mallows permutation model}\label{Sect.1.1}

As introduced above, Mallows permutation model involves the specification of a distance metric $d(\sigma,\tau)$ on permutations, which is usually used to measure the closeness between two permutations. This distance metric can be chosen from a host of metrics on permutations. We list several examples of metrics on permutations that are widely used in statistics as follows:
\begin{itemize}
    \item $L^1$ distance, or Spearman's footrule: $d(\sigma,\tau)=\sum_{i=1}^n|\sigma(i)-\tau(i)|$;
    \item $L^2$ distance, or Spearman's rank correlation: $d(\sigma,\tau)=\sum_{i=1}^n (\sigma(i)-\tau(i))^2$;
    \item Kendall's $\tau$: $d(\sigma,\tau)=$ minimum number of pairwise adjacent transpositions taking $\sigma^{-1}$ to $\tau^{-1}$;
    \item Cayley distance: $d(\sigma,\tau)=$ minimum number of transpositions taking $\sigma$ to $\tau$; 
    \item Hamming distance: $d(\sigma,\tau)=\#\{i\in\{1,\cdots,n\}:\sigma(i)\neq \tau(i)\}$;
    \item Ulam's distance: $d(\sigma,\tau)=n-$ the length of the longest increasing subsequence in $\tau\sigma^{-1}$.
\end{itemize}
Cayley distance has the following alternative expression: for any $\sigma,\tau\in S_n$,
\begin{equation*}
    d(\sigma,\tau)=n-C(\sigma\tau^{-1}),
\end{equation*}
where $C(\sigma)$ is the number of cycles of $\sigma$. The reader is referred to \cite[Chapter 6]{D} for further discussions on and statistical applications of metrics on permutations.

Mallows permutation model has been widely applied to various fields including statistics, psychology, and social science. For example, Feigin and Cohen \cite{FC} used Mallows permutation model to analyze the concordance between several judges who are asked to rank a list of objects. Critchlow \cite{Cri} extended Mallows permutation model to partially ranked data and showed that the extended model fits the ranking data well in several real examples. Further generalizations and applications of Mallows permutation model can be found in, for example, \cite{ABSV,CBBK,CFV,FV,FV2,LL,LM2,MPPB}.

In this paper, we consider Mallows permutation models with $L^1$ and $L^2$ distances. These models carry a spatial structure, and are also known as ``spatial random permutations'' in the mathematical physics literature, in connection to Feynman's approach to the study of quantum Bose gas \cite{Fey}. There has been quite some mathematical and statistical interest in such models in the literature. For example, in \cite{M1}, Mukherjee showed that when $\beta$ scales with $n$ as $\beta\sim c\slash n$ (with fixed $c$), for $\sigma$ drawn from the $L^1$ model, the random probability measure $n^{-1}\sum_{i=1}^n\delta_{(i\slash n,\sigma(i)\slash n)}$ on the unit square converges weakly in probability, and developed statistical estimation methods for the parameter $\beta$ based on the convergence results. Analogs of the results hold for the $L^2$ model when $\beta$ scales with $n$ as $\beta\sim c\slash n^2$ (again with $c$ fixed). Then in \cite{M2} he further established Poisson limit theorems for the number of cycles with fixed lengths for both the $L^1$ and $L^2$ models based on the results in \cite{M1} (with the same regimes of $\beta$). Biskup and Richthammer \cite{BR} studied Gibbs measures on permutations of a locally finite infinite set, which include infinite versions of the $L^1$ and $L^2$ models considered here for the regime when $\beta$ is fixed. Fyodorov and Muirhead \cite{FM} studied the band structure of the $L^1$ model, and established the asymptotic behavior of mean displacement per site for $\sigma$ drawn from the $L^1$ model. Here, mean displacement per site of a permutation $\sigma$ is defined as $n^{-1}\sum_{i=1}^n \mathbb{E}_{\beta}[|\sigma(i)-i|]$.

The normalizing constants for Mallows permutation models with $L^1$ and $L^2$ distances are expressed as certain permanents that are hard to compute in general. There is no known exact algorithms for sampling from these models, either. This is in contrast to many other random permutation models that have been analyzed in the literature, including Mallows permutation models with Kendall's $\tau$ or Cayley distance. Mallows permutation model with the latter distance is also known as ``Ewens sampling formula'' in the literature (see, for example, \cite{Crane}). The normalizing constants for Mallows permutation models with Kendall's $\tau$ and Cayley distance can be explicitly expressed in our notations as
\begin{equation*}
    \prod_{i=1}^n\frac{1-e^{-\beta i}}{1-e^{-\beta}} \text{ and } \frac{(e^{\beta}+1)\cdots (e^{\beta}+n-1)}{e^{\beta(n-1)}},
\end{equation*}
respectively (see e.g. \cite[Corollary 1.3.13]{Sta} and \cite{Crane} for details). For Mallows permutation model with Kendall's $\tau$, an exact sampling algorithm was originally introduced by Mallows \cite{Mal} (see also \cite[Section 2]{BP}). Mallows permutation model with Cayley distance can be exactly sampled in various ways \cite[Section 4]{Crane}, one of which is through the Chinese restaurant process (see e.g. \cite{Aldous}). These exactly solvable structures greatly facilitate the analysis of Mallows permutation models with Kendall's $\tau$ and Cayley distance; see e.g. \cite{AC,BB,BP,BDJ2,CDE,GP,GO,He,HHL,MS,M2,Pins,PT,Starr} and \cite{ABT,Crane,DS,Feray,Ka3,Ka1,Ka2} for various probabilistic properties of the two models.

In view of the lack of explicit expressions for the normalizing constants and exact sampling algorithms for the $L^1$ and $L^2$ models, we consider Markov chain-based approximate sampling algorithms for both models in this paper. That is, we design a suitable Markov chain for either of the models, and after running the chain for ``long enough time'', the distribution of the current state is close to the probability measure that corresponds to the particular model. The specific Markov chains that we use here are called \emph{hit and run algorithms}, which we will review in Section \ref{Sect.1.2} below. 

There have been several Markov chain-based algorithms for sampling from Mallows permutation models with other distances in the literature. For example, Diaconis and Hanlon \cite{DH} characterized the mixing time of a Metropolis algorithm for sampling from Mallows permutation model with Cayley distance. Later, Diaconis and Ram \cite{DR} obtained precise bounds on the mixing times of Metropolis algorithms for sampling from Mallows permutation model with Kendall's $\tau$. However, for Mallows permutation models with $L^1$ and $L^2$ distances, useful mixing time upper bounds for the corresponding Metropolis algorithms have remained open so far. As shown in Section \ref{Sect.3}, for these Mallows models, the hit and run algorithms as considered in this paper mix much faster than the corresponding Metropolis algorithms.

\subsection{Hit and run algorithms}\label{Sect.1.2}

Hit and run algorithms are a broad class of MCMC algorithms that includes variously the Swendsen-Wang algorithm, the data augmentation algorithm, and the auxiliary variables algorithm. This class of algorithms have been widely used in statistics and scientific simulations. A comprehensive overview on hit and run algorithms can be found in \cite{AD}. 

The original version of hit and run algorithms was introduced for sampling from probability densities on an Euclidean space (\cite{Turcin}; see also \cite{AD}). Let $\pi(x)$ be a probability density on $\mathbb{R}^d$. The hit and run algorithm is a Markov chain on $\mathbb{R}^d$ whose every step can be specified as follows. Starting from $x$, pick a point $y$ uniformly at random from the unit sphere centered at $x$; given $y$, pick a point $z$ on the line determined by the points $x$ and $y$ from the conditional density of $\pi(x)$ restricted to the line, and move from $x$ to $z$. The stationary distribution of this chain can be shown to be equal to $\pi$. 

Several generalized versions of hit and run algorithms have been introduced in \cite{AD}. We introduce the version that is most relevant to the examples considered here following \cite{AD}. Let $\mathcal{X}$ be a finite set, and $\pi(x)>0$ a probability measure on $\mathcal{X}$. The (generalized) hit and run algorithm is a Markov chain on $\mathcal{X}$ for sampling from $\pi$. The algorithm involves introducing a set $I$ of auxiliary variables. For each $i\in I$ and $x\in\mathcal{X}$, we also need to specify a proposal kernel $w_x(i)\geq 0$, such that the following two conditions hold:
\begin{itemize}
    \item $\sum_{i\in I}w_x(i)=1$;
    \item For each $i\in I$, there is at least one $x$ such that $w_x(i)>0$.
\end{itemize}
Let $f(x,i)=w_x(i)\pi(x)$, which is a probability measure on $\mathcal{X}\times I$. For each $i\in I$, we further specify a Markov chain $K_i(x,y)$ with stationary distribution given by the conditional distribution $f(x|i)$. 

Each step of the hit and run algorithm proceeds as follows. From $x\in\mathcal{X}$, pick $i\in I$ from the proposal kernel $w_x(i)$; given $i\in I$, run one step of the chain $K_i(x,y)$ to obtain $y$ and move from $x$ to $y$. The transition matrix of the whole algorithm is thus given by
\begin{equation*}
    K(x,y)=\sum_{i\in I} w_x(i)K_i(x,y).
\end{equation*}
It can be shown that the stationary distribution of this chain is given by $\pi$.

Many examples of hit and run algorithms can make big jumps in the state space in one step. In practice, it has been observed that such non-local chains can mix much faster than local MCMC algorithms such as the Metropolis algorithm. However, it has remained difficult to obtain useful quantitative bounds on mixing times of hit and run algorithms. In the following, we review several examples that have been analyzed in the literature. In a series of papers, Lov{\'a}sz \cite{Lovasz1} and Lov{\'a}sz \& Vempala \cite{LV1,LV2} showed that, for a version of the hit and run algorithm that samples from a log-concave density on $\mathbb{R}^d$, the chain mixes within order $d^{3+\epsilon}$ steps (for any fixed $\epsilon>0$). Diaconis and Ram \cite{DR3} introduced a version of the hit and run algorithm on partitions of a positive integer $k$ with eigenfunctions given by the coefficients of the Macdonald polynomials (when expanded in terms of power sum polynomials) and stationary distribution given by a new two-parameter family of measures on partitions. They showed that the chain mixes in a bounded number of steps for arbitrarily large $k$. Diaconis \cite{Dia} showed that an instance of the Burnside process (see \cite[Section 3.4]{AD})--a special case of hit and run algorithms--mixes in a bounded number of steps for arbitrary problem size. Later Diaconis and Zhong \cite{DZ} obtained sharp rate of convergence to the stationary distribution of this chain. Recently, Boardman, Rudolf, and Saloff-Coste \cite{BRS} analyzed a hit and run version of top-to-random shuffle, and showed that for this particular example, the hit and run algorithm accelerates mixing by at most a constant factor in $L^2$ and sup-norm.

\bigskip

The rest of this paper is organized as follows. In Section \ref{Sect.2}, we review the hit and run algorithm for sampling from the $L^2$ model following \cite{AD}, and introduce the hit and run algorithm for sampling from the $L^1$ model. We also discuss hit and run algorithms for sampling from the analogs of both the $L^1$ and $L^2$ models with weighted distances in this section. In the next section, we obtain mixing time upper bounds for hit and run algorithms for both the $L^1$ and $L^2$ models for certain ranges of the parameter $\beta$. The proofs are based on the path coupling technique. We also obtain mixing time lower bounds for the corresponding Metropolis algorithms, which suggest that the hit and run algorithms mix much faster than the Metropolis algorithms. Then in Section \ref{Sect.4}, we introduce a twisted version of the hit and run algorithm for sampling from a two-parameter permutation model that involves the $L^1$ distance and Cayley distance. The $\log{n}$ mixing time upper bound pertains to this generalization. 
The hit and run algorithms are further generalized to sample from higher dimensional analogs of both the $L^1$ and $L^2$ models in Section \ref{Sect.5}. Finally, we present some simulation studies in Section \ref{Sect.6}. The simulation results suggest that hit and run algorithms have much better empirical performance than Metropolis algorithms for the examples considered here.

We conclude this Introduction by setting up some notations. For any $n\in\mathbb{N}_{+}=\{1,2,3,\cdots\}$, we denote by $[n]:=\{1,2,\cdots,n\}$. For any finite set $A$, we let $|A|$ be the cardinality of $A$. 

\subsection{Acknowledgement}\label{Sect.1.3}
The author wishes to thank his PhD advisor, Persi Diaconis, for his encouragement, support, and many helpful conversations. The author also thanks Paul Switzer for suggesting the permutation models with weighted distances in Section \ref{Sect.2.3}. 

\section{Hit and run algorithms for Mallows permutation models with $L^1$ and $L^2$ distances}\label{Sect.2}

In this section, we review\slash introduce hit and run algorithms for sampling from Mallows permutation models with $L^1$ and $L^2$ distances. The hit and run algorithm for the $L^2$ model was originally introduced in \cite{AD}, and is reviewed in Section \ref{Sect.2.1}. The hit and run algorithm for the $L^1$ model is new, and is given in Section \ref{Sect.2.2}. Then in Section \ref{Sect.2.3}, following a suggestion by Paul Switzer, we introduce hit and run algorithms for the analogs of both the $L^1$ and $L^2$ models with weighted distances.

\subsection{Hit and run algorithm for Mallows permutation model with $L^2$ distance}\label{Sect.2.1}

In this subsection, we review the hit and run algorithm for sampling from Mallows permutation model with $L^2$ distance following \cite{AD}. We recall that the corresponding probability measure is denoted by $\tilde{\mathbb{P}}_{\beta}$. We assume that $\beta>0$ throughout this subsection.

The key observation is that the $L^2$ distance $\tilde{H}(\sigma,\tau)$ as defined in (\ref{Dis}) satisfies
\begin{equation}
    \tilde{H}(\sigma,Id)=\sum_{i=1}^n\sigma(i)^2+\sum_{i=1}^n i^2-2\sum_{i=1}^n i\sigma(i)=2\sum_{i=1}^n i^2-2\sum_{i=1}^n i\sigma(i)
\end{equation}
for any $\sigma\in S_n$. Hence we have
\begin{equation}\label{Cor}
    \tilde{\mathbb{P}}_{\beta}(\sigma)\propto e^{2\beta\sum_{i=1}^n i\sigma(i)}.
\end{equation}

Based on (\ref{Cor}), the following hit and run algorithm for sampling from $\tilde{\mathbb{P}}_{\beta}$ was introduced in \cite{AD}. The set of auxiliary variables is taken as $I=[0,\infty)^n$. Each step of the algorithm consists of the following two sequential parts:

\begin{itemize}
  \item Starting from $\sigma$, for each $i\in [n]$, independently sample $u_i$ from the uniform distribution on $[0, e^{2\beta i\sigma(i)}]$. Let $b_i=\frac{\log(u_i)}{2\beta i}$.
  \item Sample $\sigma'$ uniformly from the set $\{\tau\in S_n: \tau(i)\geq b_i\text{ for every }i\in [n]\}$, and move to the new state $\sigma'$.
\end{itemize}

The sampling problem in the second part can be efficiently implemented as follows: Look at places $i$ where $b_i\leq 1$, and place the symbol $1$ at a uniform choice among these places; look at places where $b_i\leq 2$, and place the symbol $2$ at a uniform choice among these places (with the place where the symbol $1$ was placed deleted); and so on. This gives the permutation $\sigma'$. Here, we say that the symbol $j$ is placed at the place $i$ if $\sigma'(i)=j$.

\subsection{Hit and run algorithm for Mallows permutation model with $L^1$ distance}\label{Sect.2.2}

In this subsection, we introduce the hit and run algorithm for sampling from Mallows permutation model with $L^1$ distance. We recall that the corresponding probability measure is denoted by $\mathbb{P}_{\beta}$. We assume that $\beta>0$ throughout this subsection.

In this case, the key observation is that for any $\sigma\in S_n$,
\begin{equation*}
    \sum_{i=1}^n (\sigma(i)-i)_{+}-\sum_{i=1}^n(\sigma(i)-i)_{-}=\sum_{i=1}^n \sigma(i)-\sum_{i=1}^n i=0.
\end{equation*}
This leads to the following equivalent form of $H(\sigma,Id)$ as defined in (\ref{Dis}):
\begin{equation}
    H(\sigma,Id)=\sum_{i=1}^n(\sigma(i)-i)_{+}+\sum_{i=1}^n(\sigma(i)-i)_{-}=2\sum_{i=1}^n (\sigma(i)-i)_{+}.
\end{equation}
Therefore,
\begin{equation}\label{foot}
    \mathbb{P}_{\beta}(\sigma)\propto e^{-2\beta \sum_{i=1}^n(\sigma(i)-i)_{+}}.
\end{equation}

Based on (\ref{foot}), we have the following hit and run algorithm for sampling from $\mathbb{P}_{\beta}$. We again take $I=[0,\infty)^n$ as the set of auxiliary variables. Each step of the algorithm consists of the following two sequential parts:
\begin{itemize}
  \item Starting from $\sigma$, for each $i\in [n]$, independently sample $u_i$ from the uniform distribution on $[0, e^{-2\beta (\sigma(i)-i)_{+} }]$. Let $b_i=i-\frac{\log(u_i)}{2\beta}$.
  \item Sample $\sigma'$ uniformly from the set $\{\tau\in S_n: \tau(i)\leq b_i\text{ for every }i\in [n]\}$, and move to the new state $\sigma'$.
\end{itemize}

Again, the sampling problem in the second part can be efficiently implemented: Look at places $i$ where $b_i\geq n$, and place the symbol $n$ at a uniform choice among these places; look at places where $b_i\geq n-1$, and place the symbol $n-1$ at a uniform choice among these places (with the place where the symbol $n$ was placed deleted); and so on. This gives the permutation $\sigma'$.

The following proposition validates the correctness of the above algorithm.

\begin{proposition}
Denote by $K_{\beta}$ the transition matrix of the above hit and run algorithm. Then $K_{\beta}$ is reversible with respect to $\mathbb{P}_{\beta}$.
\end{proposition}
\begin{proof}
We note that for any $\sigma\in S_n$,
\begin{equation*}
    \mathbb{P}_{\beta}(\sigma)=Z_{\beta}^{-1}e^{-2\beta\sum_{i=1}^n(\sigma(i)-i)_{+}},
\end{equation*}
where $Z_{\beta}$ is the normalizing constant.

First note that $u_i\leq 1$ for every $i\in [n]$. In this case, $\tau(i)\leq b_i$ is equivalent to $u_i\leq e^{-2\beta (\tau(i)-i)_{+}}$ for any $\tau\in S_n$. Hence for any $\sigma,\sigma'\in S_n$, we have
\begin{eqnarray*}
&&K_{\beta}(\sigma,\sigma')\\
&=&\int_{\prod_{i=1}^n[0,e^{-2\beta (\sigma(i)-i)_{+}}]} \frac{1_{\sigma'(i)\leq b_i, \forall i\in[n]}}{|\{\tau\in S_n: \tau(i)\leq b_i, \forall i\in[n]\}|} d u_1 \cdots d u_n\\
&&\times e^{2\beta \sum_{i=1}^n (\sigma(i)-i)_{+}}\\
&=& \int_{\prod_{i=1}^n[0,e^{-2\beta (\sigma(i)-i)_{+}}]}  \frac{1_{u_i\leq e^{-2\beta (\sigma'(i)-i)_{+}},\forall i\in[n]}}{|\{\tau\in S_n: u_i\leq e^{ -2 \beta(\tau(i)-i)_{+}}, \forall i\in[n]\}|}d u_1 \cdots d u_n\\
&&\times e^{2\beta \sum_{i=1}^n (\sigma(i)-i)_{+}}\\
&=& \int_{B(\beta,\sigma,\sigma')} \frac{1}{|\{\tau\in S_n: u_i\leq e^{ -2 \beta(\tau(i)-i)_{+}}, \forall i\in [n]\}|} d u_1 \cdots d u_n\\
&&\times 
 e^{2\beta \sum_{i=1}^n (\sigma(i)-i)_{+}},
\end{eqnarray*}
where $B(\beta,\sigma,\sigma'):=\prod_{i=1}^n[0,e^{-2\beta \max\{(\sigma(i)-i)_{+},(\sigma'(i)-i)_{+}\}}]$, and $b_i=i-\frac{\log(u_i)}{2\beta}$ for every $i\in [n]$.

Therefore, we have
\begin{eqnarray*}
&&\mathbb{P}_{\beta}(\sigma)K_{\beta}(\sigma,\sigma')\\
&=& \int_{B(\beta,\sigma,\sigma')} \frac{1}{|\{\tau\in S_n: u_i\leq e^{ -2 \beta(\tau(i)-i)_{+}},\forall i\in [n]\}|} d u_1 \cdots d u_n\\
&&  \times Z_{\beta}^{-1}. 
\end{eqnarray*}
Similarly,
\begin{eqnarray*}
&&\mathbb{P}_{\beta}(\sigma')K_{\beta}(\sigma',\sigma)\\
&=& \int_{B(\beta,\sigma,\sigma')} \frac{1}{|\{\tau\in S_n: u_i\leq e^{ -2 \beta(\tau(i)-i)_{+}}\forall i\in [n]\}|} d u_1 \cdots d u_n\\
&& \times Z_{\beta}^{-1}.
\end{eqnarray*}
Therefore
\begin{equation}
    \mathbb{P}_{\beta}(\sigma)K_{\beta}(\sigma,\sigma')=\mathbb{P}_{\beta}(\sigma')K_{\beta}(\sigma',\sigma).
\end{equation}
The above equation shows that the chain $K_{\beta}$ is reversible with respect to $\mathbb{P}_{\beta}$.

\end{proof}

\subsection{Hit and run algorithms for generalized models with weighted distances}\label{Sect.2.3}

Following a suggestion by Paul Switzer \cite{PS}, we consider the following generalizations of Mallows permutation models with $L^1$ and $L^2$ distances, which allow more flexibility in modeling permutation data. The generalized models involve a monotone increasing function $w: [n]\rightarrow \mathbb{R}_{+}$ (the case when $w$ is monotone decreasing can also be treated in a similar manner). The generalized models are given by the following probability measures on $S_n$:
\begin{equation}
    \mathbb{P}_{\beta,w}(\sigma) \propto \exp(-\beta\sum_{i=1}^n |w(\sigma(i))-w(i)|),
\end{equation}
\begin{equation}
    \tilde{\mathbb{P}}_{\beta,w}(\sigma)\propto \exp(-\beta\sum_{i=1}^n (w(\sigma(i))-w(i))^2),
\end{equation}
where $\sigma\in S_n$. We assume that $\beta>0$ throughout this subsection.

We note that
\begin{equation*}
    \sum_{i=1}^n (w(\sigma(i))-w(i))_{+}-\sum_{i=1}^n (w(\sigma(i))-w(i))_{-}=\sum_{i=1}^n (w(\sigma(i))-w(i))=0,
\end{equation*}
hence
\begin{eqnarray*}
    \sum_{i=1}^n |w(\sigma(i))-w(i)|&=&\sum_{i=1}^n (w(\sigma(i))-w(i))_{+}+\sum_{i=1}^n (w(\sigma(i))-w(i))_{-}\\
    &=&2\sum_{i=1}^n (w(\sigma(i))-w(i))_{+}.
\end{eqnarray*}
This leads to
\begin{equation}
    \mathbb{P}_{\beta,w}(\sigma) \propto \exp(-2\beta\sum_{i=1}^n (w(\sigma(i))-w(i))_{+}),
\end{equation}
for any $\sigma\in S_n$.

Similarly,
\begin{equation*}
    \sum_{i=1}^n (w(\sigma(i))-w(i))^2=2\sum_{i=1}^n w(i)^2-2\sum_{i=1}^n w(i)w(\sigma(i)),
\end{equation*}
which leads to
\begin{equation}
    \tilde{\mathbb{P}}_{\beta,w}(\sigma)\propto \exp(2\beta\sum_{i=1}^n w(i)w(\sigma(i))),
\end{equation}
for any $\sigma\in S_n$.

Following the procedure in Section \ref{Sect.2.2}, we have the following hit and run algorithm for sampling from $\mathbb{P}_{\beta,w}$:
\begin{itemize}
    \item Starting from $\sigma$, for every $i\in [n]$, sample $u_i$ independently from the uniform distribution on $[0,e^{-2\beta(w(\sigma(i))-w(i))_{+}}]$. Let $b_i=w(i)-\frac{\log(u_i)}{2\beta}$ for each $i\in [n]$.
    \item For every $i\in [n]$, let $q_i=\max\{j\in [n]:w(j)\leq b_i\}$. Sample $\tau\in S_n$ uniformly from the set $\{\tau\in S_n: \tau(i)\leq q_i\text{ for every }i   \in    [n]\}$. Then move to the new state $\tau$.
\end{itemize}

Similarly, following the procedure in  Section \ref{Sect.2.1}, we also have the following hit and run algorithm for sampling from $\tilde{\mathbb{P}}_{\beta,w}$:
\begin{itemize}
    \item Starting from $\sigma$, for every $i\in [n]$, sample $u_i$ independently from the uniform distribution on $[0,e^{2\beta w(i) w(\sigma(i))}]$. Let $b_i=\frac{\log(u_i)}{2\beta w(i)}$ for each $i\in [n]$. 
    \item For every $i\in [n]$, let $q_i=\min\{j\in [n]: w(j)\geq b_i\}$. Sample $\tau\in S_n$ uniformly from the set $\{\tau\in S_n: \tau(i)\geq q_i\text{ for every }i\in [n]\}$, and move to the new state $\tau$.  
\end{itemize}

\bigskip

We end this section with a final remark. The $L^1$ and $L^2$ models in this paper can be considered as special cases of the following $L^p$ model for $p>0$:
\begin{equation*}
    \mathbb{P}_{\beta,p}(\sigma)\propto \exp(-\beta\sum_{i=1}^n |\sigma(i)-i|^p),
\end{equation*}
where $\sigma\in S_n$. The hit and run algorithms as introduced above make use of certain special structures of the model when $p\in \{1,2\}$, and do not work for $p\notin\{1,2\}$.  

\section{Mixing time analysis of the hit and run algorithms}\label{Sect.3}
In this section, we analyze the mixing times of the hit and run algorithms for Mallows permutation models with $L^1$ and $L^2$ distances as introduced in Section \ref{Sect.2}. The main tool is the path coupling technique as introduced in \cite{BD}. In Section \ref{Sect.3.1} we review the background on mixing time and the path coupling technique. In Sections \ref{Sect.3.2}-\ref{Sect.3.3}, we use path coupling to obtain mixing time upper bounds for the two hit and run algorithms for certain ranges of parameters that are of practical interest. We also present mixing time lower bounds for the Metropolis algorithms sampling from both models, which show that the hit and run algorithms considered here mix much faster than the corresponding Metropolis algorithms.

\subsection{Background on mixing time and path coupling}\label{Sect.3.1}

In this subsection, we review the background on mixing time and the path coupling technique following \cite[Chapters 4 and 14]{LP}. The reader is referred to \cite{BD} and \cite[Chapter 14]{LP} for further details on path coupling. 

Consider a Markov chain $X_t$ on a finite state space $\Omega$ with stationary distribution $\pi$. For any $x\in\Omega$, we denote by $K^t_x$ the distribution of the chain at the $t$th step started from the state $x$. The total variation distance between the two probability measures $K^t_x$ and $\pi$ is defined as
\begin{equation*}
    \|K^t_x-\pi\|:=\frac{1}{2}\sum_{y\in\Omega}|K^t_x(y)-\pi(y)|=\frac{1}{2}\sum_{y\in\Omega}|K^t(x,y)-\pi(y)|.
\end{equation*}
Let $d(t):=\max_{x\in\Omega}\|K^t_x-\pi\|$. For any $\epsilon\in (0,1)$, the $\epsilon$-mixing time of the Markov chain, denoted by $t_{mix}(\epsilon)$, is defined as 
\begin{equation*}
    t_{mix}(\epsilon):=\min\{t: d(t)\leq \epsilon\}.
\end{equation*}

Path coupling is a useful technique for upper bounding mixing times of Markov chains. Assume that the state space $\Omega$ is the vertex set of a simple, undirected, connected graph $G=(\Omega,E)$ and that each edge $\{u,v\}\in E$ is assigned a length $l(u,v)\geq 1$. For any $u,v\in \Omega$, a path from $u$ to $v$ is a sequence of states $P=(x_0,x_1,\cdots,x_s)$ such that $x_0=u,x_s=v$ and $\{x_i,x_{i+1}\}\in E$ for each $i=0,1,\cdots,s-1$. The length of the path is defined as $\sum_{i=0}^{s-1}l(x_i,x_{i+1})$. The path metric on $\Omega$ is then defined as
\begin{equation}\label{Pm}
    \rho(u,v)=\min\{\text{length of }P: P\text{ is a path connecting }u\text{ to }v\}.
\end{equation}

The key result for path coupling is the following

\begin{theorem}[\cite{BD}]\label{Path}
Suppose that $\Omega,E,l,\rho$ are given as in the preceding. Denote by $diam(\Omega)=\max_{u,v\in\Omega} \rho(u,v)$. Suppose further that for each edge $\{u,v\}\in E$, there exists a coupling $(U,V)$ of the distributions $K^1_u$ and $K^1_v$ such that
\begin{equation}
    \mathbb{E}_{u,v}[\rho(U,V)]\leq \rho(u,v)e^{-\alpha},
\end{equation}
where $\alpha>0$. Then we have for any $t\geq 0$,
\begin{equation}
    \max_{u\in\Omega}\|K_u^t-\pi\| \leq e^{-\alpha t}  diam(\Omega).
\end{equation}
\end{theorem}

\subsection{Mallows permutation model with $L^2$ distance}\label{Sect.3.2}

In this subsection, we use the path coupling technique to prove an order $\log{n}$ mixing time upper bound for the hit and run algorithm for sampling from Mallows permutation model with $L^2$ distance (as introduced in Section \ref{Sect.2.1}) for $0<\beta\leq 1\slash (16cn^2)$ with any fixed $c>1$. We denote by $\tilde{K}_{\beta}$ the transition matrix of this hit and run algorithm. For any permutation $\sigma_0\in S_n$, we denote by $\tilde{K}^l_{\beta,\sigma_0}$ the distribution of the chain after $l$ steps from the starting state $\sigma_0$. Recall that $\tilde{\mathbb{P}}_{\beta}$ is the stationary distribution of the chain.

The main result is the following

\begin{theorem}\label{rho}
Assume that $0<\beta<1\slash (16n^2)$. Then for any $l\geq 0$,
\begin{equation}
    \max_{\sigma_0\in S_n}\|\tilde{K}_{\beta,\sigma_0}^l-\tilde{\mathbb{P}}_{\beta}\|\leq 2n^2(16\beta n^2)^l.
\end{equation}
This implies that
\begin{equation}
    t_{mix}(\frac{1}{4})\leq \Big\lceil\frac{\log(8 n^2)}{-\log(16\beta n^2)}\Big\rceil.
\end{equation}
\end{theorem}
\begin{remark}
If $0<\beta\leq 1\slash(16cn^2)$ for some fixed $c>1$, then
\begin{equation*}
    t_{mix}(\frac{1}{4})\leq \Big\lceil\frac{3\log(2n)}{\log(c)}\Big\rceil.
\end{equation*}
This gives an order $\log{n}$ mixing time upper bound of the hit and run algorithm.
\end{remark}
\begin{remark}
By \cite{M1,M2}, for $\beta\sim c\slash n^2$ with fixed $c>0$, random permutations drawn from $\tilde{\mathbb{P}}_{\beta}$ converge to a non-trivial permutation limit as $n\rightarrow\infty$, and the limiting distribution of the number of fixed points (or cycles of fixed length) is different from that of uniform random permutations. Therefore, the behavior of a random permutation drawn from $\tilde{\mathbb{P}}_{\beta}$ under such scaling is quite different from that of a uniform random permutation.
\end{remark}

Theorem \ref{rho} is proved using the path coupling technique. Note that the state space $\Omega=S_n$. The edge set is specified as follows: For any two permutations $\sigma,\tau\in S_n$, we connect them with an edge if and only if $\tau=(i,i+1)\sigma$ for some $i\in\{1,\cdots , n-1\}$. This gives a graph $G=(\Omega,E)$. For every edge $\{u,v\}\in E$, we let $l(u,v)=1$. We further let $\rho$ be the path metric on $\Omega$ as defined by (\ref{Pm}).

We first establish the following lemma on permutations with one-sided restrictions. The combinatorics of permutations with one-sided restrictions is well studied. \cite{DGH} derives the distribution of various metrics under the uniform distribution. \cite{Blum} studies the running time for the Metropolis algorithm.

\begin{lemma}\label{Lem}
Suppose that $k\in [n]$, $b_k'\in [n]$, and $b_i\in [n]$ for every $i\in [n]$. Define 
\begin{equation*}
    \mathcal{S}:=\{\sigma\in S_n: \sigma(i)\geq b_i\text{ for every }i\in [n]\},
\end{equation*}
\begin{equation*}
    \mathcal{S}':=\{\sigma\in S_n: \sigma(i)\geq b_i\text{ for every }i\in [n]\backslash \{k\}, \sigma(k)\geq b_k'\}.
\end{equation*}
Assume that $\mathcal{S},\mathcal{S}'\neq \emptyset$. Further let $\mu,\mu'$ be the uniform distributions on $\mathcal{S},\mathcal{S}'$, respectively. Then there exists a coupling $(X,X')$ of $\mu$ and $\mu'$, such that 
\begin{equation}
    \rho(X,X')\leq 2n,
\end{equation}
where the path metric $\rho=\rho(u,v)$ is defined in (\ref{Pm}).
\end{lemma}
\begin{proof}
The proof consists of two parts. In the first part, we construct the coupling $(X,X')$. Then in the second part, we obtain the desired upper bound on $\rho(X,X')$.

\paragraph{Part 1}

Without loss of generality we assume that $b_k'< b_k$ (note that if $b_k'=b_k$, then $\mu'=\mu$, and the conclusion holds trivially). For every $l\in [n]$, let
\begin{equation*}
    N_l=\#\{i\in [n]: b_i\leq l\}-l+1.
\end{equation*}

We construct the coupling as follows. Sequentially for $l=1,\cdots,b_k'-1$, we look at places $i$ with $b_i\leq l$ that have not yet been taken, and let $X^{-1}(l)=X'^{-1}(l)$ be a uniform choice from these available places. Note that $X^{-1}(l)=X'^{-1}(l)$ for every $l\in\{1,\cdots,b_k'-1\}$.

When $l=b_k'$, there are $N_l$ possible choices of $X^{-1}(l)$, which we denote by $j_{l,1},\cdots,j_{l,N_l}$, and $N_l+1$ possible choices $j_{l,1},\cdots,j_{l,N_l},k$ of $X'^{-1}(l)$. We define the coupling at this step as follows. For every $s\in\{1,\cdots,N_l\}$, we let $X^{-1}(l)=X'^{-1}(l)=j_{l,s}$ with probability $\frac{1}{N_l+1}$; with probability $\frac{1}{N_l+1}$, we let $X'^{-1}(l)=k$ and let $X^{-1}(l)$ be a uniform choice among $j_{l,1},\cdots,j_{l,N_l}$. Note that for either case, if $l+1<b_k$, the set of possible choices of $X'^{-1}(l+1)$ is the union of set of possible choices of $X^{-1}(l+1)$ and another element (one of $j_{l,1},\cdots,j_{l,N_l},k$).

Now we define the coupling for $b_k'+1\leq l\leq b_k-1$ inductively. Assume that the set of possible $N_l+1$ choices of $X'^{-1}(l)$ is the union of the set of possible $N_l$ choices of $X^{-1}(l)$ (denoted by $j_{l,1},\cdots,j_{l,N_l}$) and another element (denoted by $t_l$). Note that this is satisfied for $l=b_k'+1$ if $b_k'+1\leq b_k-1$. For every $s\in\{1,\cdots,N_l\}$, we let $X^{-1}(l)=X'^{-1}(l)=j_s$ with probability $\frac{1}{N_l+1}$; with probability $\frac{1}{N_l+1}$, we let $X'^{-1}(l)=t_l$ and let $X^{-1}(l)$ be a uniform choice among $j_{l,1},\cdots,j_{l,N_l}$. Note that if $l+1\leq b_k-1$, then the assumption is satisfied with $l$ replaced by $l+1$. Therefore, the above procedure gives a valid coupling for $b_k'+1\leq l\leq b_k-1$. 

When $l=b_k$, suppose that $j_{l,1},\cdots,j_{l,N_l-1},k$ are the $N_l$ possible choices of $X^{-1}(l)$. Note that $X'^{-1}(l)$ also has $N_l$ possible choices, which are given by $j_{l,1},\cdots,j_{l,N_l-1},t$ for some $t\notin\{j_{l,1},\cdots,j_{l,N_l-1}\}$ by our construction. For every $s\in\{1,\cdots,N_l-1\}$, we let $X^{-1}(l)=X'^{-1}(l)=j_{l,s}$ with probability $\frac{1}{N_l}$; we also let $X^{-1}(l)=k$ and $X'^{-1}(l)=t$ with probability $\frac{1}{N_l}$. 

We define the coupling for $l=b_k+1,\cdots,n$ inductively as follows. Assume that either of the following two possibilities is true: (a) the possible $N_{l}$ choices of $X^{-1}(l)$ and $X'^{-1}(l)$ are the same; (b) the $N_l$ possible choices of $X^{-1}(l)$ are $j_{l,1},\cdots,j_{l,N_l-1},k$, and the $N_l$ possible choices of $X'^{-1}(l)$ are $j_{l,1},\cdots,j_{l,N_l-1},t$ for some $t\notin \{j_{l,1},\cdots,j_{l,N_l-1}\}$. Note that the assumption holds when $l=b_k+1$. For case (a), we let $X^{-1}(l)=X'^{-1}(l)$ be a uniform choice from these choices. For case (b), for every $s\in\{1,\cdots,N_l-1\}$, we let $X^{-1}(l)=X'^{-1}(l)=j_{l,s}$ with probability $\frac{1}{N_l}$, and let $X^{-1}(l)=k$ and $X'^{-1}(l)=t$ with probability $\frac{1}{N_l}$. Note that if $l+1\leq n$, then the assumption is still satisfied with $l$ replaced by $l+1$. This finishes our coupling.

\paragraph{Part2}

We argue below that $X'^{-1}$ can be obtained from $X^{-1}$ by at most $2n$ adjacent transpositions. Order the elements of the set $\{l\in [n]: X^{-1}(l)\neq X'^{-1}(l)\}$ as $l_1<l_2<\cdots<l_q$. By construction, $X'^{-1}(l_1)=k$, $X^{-1}(l_q)=k$, and $X^{-1}(l_s)=X'^{-1}(l_{s+1})$ for any $s=1,\cdots,q-1$. Hence we have that
\begin{eqnarray*}
    X^{-1} &=& X'^{-1} (l_1,l_2) \cdots (l_{q-1},l_q).
\end{eqnarray*}
Note that $(l_s,l_{s+1})=(l_s,l_s+1)\cdots (l_{s+1}-2,l_{s+1}-1)(l_{s+1}-1,l_{s+1})(l_{s+1}-2,l_{s+1}-1)\cdots (l_s,l_s+1)$ for any $1\leq s\leq q-1$. Therefore, $X'^{-1}$ can be obtained from $X^{-1}$ by at most
\begin{equation}
    \sum_{s=1}^{q-1}  2(l_{s+1}-l_s)\leq 2n
\end{equation}
adjacent transpositions. Hence
\begin{equation}
    \rho(X,X')\leq 2n.
\end{equation}

\end{proof}

\begin{proof}[Proof of Theorem \ref{rho}]
We use the path coupling technique. The choice of the edge set and the path metric is given as in the preceding. Note that $diam(\Omega)\leq 2n^2$.

For any two permutations $\sigma,\sigma'\in S_n$ that are adjacent in the graph $G=(\Omega,E)$, we define the coupling as follows. Assume that $\sigma'=(i,i+1)\sigma$, where $i\in\{1,\cdots,n-1\}$. Suppose that $\sigma(k_1)=i$ and $\sigma(k_2)=i+1$. Then $\sigma'(k_1)=i+1$ and $\sigma'(k_2)=i$. Below we will construct a coupling $(X,X')$ of the distributions $\tilde{K}_{\beta,\sigma}^1$ and $\tilde{K}_{\beta,\sigma'}^1$.

Consider the first part of the hit and run algorithm. For every $j\in [n]$, let $u_j$ be sampled from the uniform distribution on $[0,e^{2\beta j\sigma(j)}]$ and $b_j=\frac{\log(u_j)}{2\beta j}$. For every $j\in [n]$, let $B_j=\max\{\lceil b_j \rceil,1\}$, and let $B_j'$ be the corresponding quantity for $\sigma'$. It suffices to construct the coupling between $\{B_j\}_{j=1}^n$ and $\{B_j'\}_{j=1}^n$. 

Note that for every $k=2,\cdots,\sigma(j)$, 
\begin{equation}
    \mathbb{P}(B_j=k)=\mathbb{P}(k-1<b_j\leq k)=(1-e^{-2\beta j})  e^{2\beta j(k-\sigma(j))}.
\end{equation}
Moreover,
\begin{equation}
    \mathbb{P}(B_j=1)=\mathbb{P}(b_j\leq 1)=e^{2\beta  j (1-\sigma(j))}.
\end{equation}
For $j\in [n]$ with $j\neq k_1,k_2$, we let $B_j'=B_j$. 

For $j=k_1$, we construct the coupling between $B_{k_1}$ and $B_{k_1}'$ as below. We let $B_{k_1}=B_{k_1}'=1$ with probability $e^{-2\beta k_1 i}$. For every $k\in\{2,\cdots,i\}$, we let $B_{k_1}=B_{k_1}'=k$ with probability $(1-e^{-2\beta k_1})e^{2\beta k_1(k-i-1)}$. We also let $B_{k_1}=1$ and $B_{k_1}'=i+1$ with probability $e^{2\beta k_1(1-i)}(1-e^{-2\beta k_1})$. For every $k\in\{2,\cdots,i\}$, we let $B_{k_1}=k$ and $B_{k_1}'=i+1$ with probability $(1-e^{-2\beta k_1})^2 e^{2\beta k_1 (k-i)}$.

For $j=k_2$, we construct the coupling between $B_{k_2}$ and $B_{k_2}'$ similarly. We let $B_{k_2}=B_{k_2}'=1$ with probability $e^{-2\beta k_2 i}$. For every $k\in\{2,\cdots,i\}$, we let $B_{k_2}=B_{k_2}'=k$ with probability $(1-e^{-2\beta k_2})e^{2\beta k_2(k-i-1)}$. We also let $B'_{k_2}=1$ and $B_{k_2}=i+1$ with probability $e^{2\beta k_2(1-i)}(1-e^{-2\beta k_2})$. For every $k\in\{2,\cdots,i\}$, we let $B'_{k_2}=k$ and $B_{k_2}=i+1$ with probability $(1-e^{-2\beta k_2})^2 e^{2\beta k_2 (k-i)}$.

It can be checked that this gives a valid coupling between $\{B_j\}_{j=1}^n$ and $\{B'_j\}_{j=1}^n$. We also have
\begin{equation}
    \mathbb{P}(B_{k_1}\neq B_{k_1}')=1-e^{-2\beta k_1},\quad \mathbb{P}(B_{k_2}\neq B_{k_2}')=1-e^{-2\beta k_2}.
\end{equation}
By the union bound,
\begin{equation}
    \mathbb{P}(B_j\neq B_j'\text{ for some }j\in [n])\leq 2-e^{-2\beta k_1}-e^{-2\beta k_2}\leq 2\beta(k_1+k_2).
\end{equation}

Now we consider the second part of the hit-and-run algorithm. Let
\begin{equation*}
    \mathcal{S}=\{\tau\in S_n:\tau(j)\geq B_j\text{ for every }j\in [n]\},
\end{equation*}
\begin{equation*}
    \mathcal{S}'=\{\tau\in S_n:\tau(j)\geq B_j'\text{ for every }j\in [n]\}
\end{equation*}
Note that
\begin{equation*}
    \mathcal{S}=\{\tau\in S_n:\tau(j)\geq b_j\text{ for every }j\in [n]\},
\end{equation*}
\begin{equation*}
    \mathcal{S}'=\{\tau\in S_n:\tau(j)\geq b_j'\text{ for every }j\in [n]\}.
\end{equation*}

If $B_j=B_j'$ for every $j\in [n]$, then $\mathcal{S}=\mathcal{S}'$. We let $X=X'$ be the uniform distribution on $\mathcal{S}$. In this case we have
\begin{equation}
    \rho(X,X')=0.
\end{equation}

Otherwise, either $B_{k_1}'\neq B_{k_1}$ or $B_{k_2}'\neq B_{k_2}$. Let 
\begin{equation*}
    \mathcal{S}''=\{\tau\in S_n:\tau(j)\geq B_j\text{ for every }j\in [n]\backslash \{k_2\}, \text{ and }\tau(k_2)\geq B'_{k_2}\}.
\end{equation*}
By Lemma \ref{Lem}, there exists a coupling $(X_1,Z_1)$ of the uniform distributions on $\mathcal{S}$ and $\mathcal{S}''$, such that
\begin{equation}
    \rho(X_1,Z_1)\leq 2n.
\end{equation}
Similarly, there exists a coupling $(Z_2,Y_2)$ of the uniform distributions on $\mathcal{S}''$ and $\mathcal{S}'$, such that
\begin{equation}
    \rho(Z_2,Y_2)\leq 2n.
\end{equation}
By the proof of \cite[Lemma 14.3]{LP}, there exists a coupling $(X,X')$ of the uniform distributions on $\mathcal{S}$ and $\mathcal{S}'$, such that
\begin{equation}
    \rho(X,X')\leq 4n.
\end{equation}

It can be checked that $(X,X')$ gives a valid coupling of $\tilde{K}_{\beta,\sigma}^1$ and $\tilde{K}_{\beta,\sigma'}^1$. Moreover,
\begin{equation}
    \mathbb{E}[\rho(X,X')]\leq 4n \mathbb{P}( B_j \neq B_j'  \text{ for some }j\in [n])\leq 8\beta n(k_1+k_2)\leq 16\beta n^2.
\end{equation}
Therefore by Theorem \ref{Path}, we have that for any $l\geq 0$,
\begin{equation}
    \max_{\sigma_0\in S_n}\|\tilde{K}_{\beta,\sigma_0}^l-\tilde{\mathbb{P}}_{\beta}\|\leq 2n^2 (16\beta n^2)^l.
\end{equation}

\end{proof}

\subsection{Mallows permutation model with $L^1$ distance}\label{Sect.3.3}

In this subsection, again by using the path coupling technique, we prove an order $\log{n}$ mixing time upper bound for the hit and run algorithm sampling from Mallows permutation model with $L^1$ distance (as introduced in Section \ref{Sect.2.2}) for $0<\beta\leq 1\slash (16cn)$ with any fixed $c>1$. The proof strategy is similar to that in Section \ref{Sect.3.2}.

Recall that we denote by $K_{\beta}$ the transition matrix of the hit and run algorithm, and that the stationary distribution of the chain is $\mathbb{P}_{\beta}$.  We also denote by $K_{\beta,\sigma_0}^l$ the distribution of the chain after $l$ steps started from the state $\sigma_0$. 

We have the following main result.

\begin{theorem}\label{footru}
Assume that $0<\beta< 1\slash (16n)$. Then for any $l\geq 0$,
\begin{equation}
    \max_{\sigma_0\in S_n}\|K_{\beta,\sigma_0}^l-\mathbb{P}_{\beta}\|\leq 2n^2 (16\beta n)^l.
\end{equation}
This implies that
\begin{equation}
    t_{mix}(\frac{1}{4})\leq \Big\lceil  \frac{\log(8 n^2)}{-\log(16 \beta n)}\Big\rceil.
\end{equation}
\end{theorem}
\begin{remark}
If $0<\beta\leq  1\slash (16 cn)$ for some fixed $c>1$, then
\begin{equation*}
    t_{mix}(\frac{1}{4})\leq \Big\lceil\frac{3\log(2n)}{\log(c)} \Big\rceil.
\end{equation*}
This gives an order $\log{n}$ mixing time upper bound of the hit and run algorithm.
\end{remark}
\begin{remark}
By \cite{M1,M2}, for $\beta\sim c\slash n$ with fixed $c>0$, random permutations drawn from $\mathbb{P}_{\beta}$ converge to a non-trivial permutation limit as $n\rightarrow\infty$, and the limiting distribution of the number of fixed points (or cycles of fixed length) is different from that of uniform random permutations. Therefore, the behavior of a random permutation drawn from $\mathbb{P}_{\beta}$ under such scaling is quite different from that of a uniform random permutation.
\end{remark}

We first present the following lemma, which can be proved in the same way as Lemma \ref{Lem}.

\begin{lemma}\label{L2}
Suppose that $k\in [n]$, $b_k'\in [n]$, and $b_i\in [n]$ for every $i\in [n]$. Define 
\begin{equation*}
    \mathcal{S}:=\{\sigma\in S_n: \sigma(i)\leq b_i\text{ for every }i\in [n] \},
\end{equation*}
\begin{equation*}
    \mathcal{S}':=\{\sigma\in S_n: \sigma(i)\leq b_i\text{ for every }i\in [n]\backslash\{k\}, \sigma(k)\leq b_k'\}.
\end{equation*}
Assume that $\mathcal{S},\mathcal{S}'\neq \emptyset$. Further let $\mu,\mu'$ be the uniform distributions on $\mathcal{S},\mathcal{S}'$, respectively. Then there exists a coupling $(X,X')$ of $\mu$ and $\mu'$, such that 
\begin{equation}
    \rho(X,X')\leq 2n,
\end{equation}
where the path metric $\rho=\rho(u,v)$ is defined in (\ref{Pm}) with the graph $G=(\Omega,E)$ and the length function $l(u,v)$ defined as in Section \ref{Sect.3.2}.
\end{lemma}

\begin{proof}[Proof of Theorem \ref{footru}]
We use the path coupling technique again. The choice of the graph structure $G=(\Omega,E)$ and the path metric $\rho$ is the same as that from the proof of Theorem \ref{rho}. We have $diam(\Omega)\leq 2n^2$.

Consider any two permutations $\sigma,\sigma'\in  S_n$ that are adjacent in the graph $G$. Assume that $\sigma'=(i,i+1)\sigma$, where $i\in\{1,\cdots,n-1\}$. Suppose that $\sigma(k_1)=i$ and $\sigma(k_2)=i+1$. Thus $\sigma'(k_1)=i+1$ and $\sigma'(k_2)=i$.
In the following, we construct a coupling $(X,X')$ of the distributions $K_{\beta,\sigma}^1$ and $K_{\beta,\sigma'}^1$.

Consider the first step of the hit and run algorithm. For each $j\in [n]$, let $u_j$ be independently sampled from the uniform distribution on $[0,e^{-2\beta (\sigma(j)-j)_{+}}]$, and let $b_j=j-\frac{\log(u_j)}{2\beta}$. We further let $B_j=\min\{\lfloor b_j\rfloor, n\}$, and let $B_j'$ be the corresponding quantity for $\sigma'$. It suffices to construct the coupling for $\{B_j\}_{j=1}^n$ and $\{B_j'\}_{j=1}^n$.

Note that for any $j\in [n]$, if $\sigma(j)\leq j$, then for any $k\in\{j,\cdots, n-1\}$,
\begin{equation}
    \mathbb{P}(B_j=k)=(1-e^{-2\beta})e^{-2\beta(k-j)},
\end{equation}
\begin{equation}
    \mathbb{P}(B_j=n)=e^{-2\beta (n-j)}.
\end{equation}
If $\sigma(j)\geq j+1$, then for any $k\in\{\sigma(j),\cdots, n-1\}$,
\begin{equation}
    \mathbb{P}(B_j=k)=(1-e^{-2\beta})e^{-2\beta(k-\sigma(j))},
\end{equation}
\begin{equation}
    \mathbb{P}(B_j=n)=e^{-2\beta (n-\sigma(j))}.
\end{equation}
For any $j\in [n]$ with $j\neq k_1,k_2$, we let $B_j'=B_j$.

For $j=k_1$, we construct the coupling between $B_{k_1}$ and $B_{k_1}'$ as follows. There are two possible cases: (a) $i\leq i+1\leq k_1$; (b) $k_1\leq i\leq i+1$. For case (a), we can couple $B_{k_1}$ and $B_{k_1}'$ such that $B_{k_1}=B'_{k_1}$. For case (b), for every $k\in\{i+1,\cdots,n-1\}$, with probability $(1-e^{-2\beta})e^{-2\beta (k-i)}$, we let $B_{k_1}=B_{k_1}'=k$. With probability $e^{-2\beta(n-i)}$, we let $B_{k_1}=B_{k_1}'=n$. For every $k\in\{i+1,\cdots,n-1\}$, with probability $(1-e^{-2\beta})^2 e^{-2\beta(k-i-1)}$, we let $B_{k_1}=i$ and $B'_{k_1}=k$. With probability $(1-e^{-2\beta})e^{-2\beta (n-i-1)}$, we take $B_{k_1}=i$ and $B'_{k_1}=n$.

For $j=k_2$, we construct the coupling between $B_{k_2}$ and $B_{k_2}'$ similarly as follows. There are two possible cases: (a) $i\leq i+1\leq k_2$; (b) $k_2\leq i\leq i+1$. For case (a), we can couple $B_{k_2}$ and $B_{k_2}'$ such that $B_{k_2}=B_{k_2}'$. For case (b), for every $k\in\{i+1,\cdots,n-1\}$, with probability $(1-e^{-2\beta})e^{-2\beta (k-i)}$, we let $B_{k_2}=B_{k_2}'=k$. With probability $e^{-2\beta(n-i)}$, we let $B_{k_2}=B_{k_2}'=n$. For every $k\in\{i+1,\cdots,n-1\}$, with probability $(1-e^{-2\beta})^2 e^{-2\beta(k-i-1)}$, we let $B_{k_2}'=i$ and $B_{k_2}=k$. With probability $(1-e^{-2\beta})e^{-2\beta (n-i-1)}$, we take $B_{k_2}'=i$ and $B_{k_2}=n$.

It can be checked that this gives a valid coupling between $\{B_j\}_{j=1}^n$ and $\{B_j'\}_{j=1}^n$. Also
\begin{equation}
    \mathbb{P}(B_{k_1}\neq B'_{k_1})=1-e^{-2\beta}, \quad \mathbb{P}(B_{k_2}\neq B'_{k_2})=1-e^{-2\beta}.
\end{equation}
By the union bound,
\begin{equation}
    \mathbb{P}(B_j\neq B_j'\text{ for some } j\in [n])\leq 2(1-e^{-2\beta})\leq 4\beta.
\end{equation}

Now we consider the second part of the hit and run algorithm. Let
\begin{equation*}
    \mathcal{S}=\{\tau\in S_n:\tau(j)\leq B_j\text{ for every }j\in[n]\},
\end{equation*}
\begin{equation*}
    \mathcal{S}'=\{\tau\in S_n:\tau(j)\leq B_j'\text{ for every }j\in [n]\}.
\end{equation*}
Note that 
\begin{equation*}
    \mathcal{S}=\{\tau\in S_n:\tau(j)\leq b_j\text{ for every }j\in [n]\},
\end{equation*}
\begin{equation*}
    \mathcal{S}'=\{\tau\in S_n:\tau(j)\leq b_j'\text{ for every }j\in [n]\}.
\end{equation*}

If $B_j=B_j'$ for every $j\in [n]$, then $\mathcal{S}=\mathcal{S}'$. We let $X=X'$ be the uniform distribution on $\mathcal{S}$. In this case we have
\begin{equation}
    \rho(X,X')=0.
\end{equation}

Otherwise, either $B_{k_1}'\neq B_{k_1}$ or $B_{k_2}'\neq B_{k_2}$. By Lemma \ref{L2}, similar to the proof of Theorem \ref{rho}, we can deduce that there exists a coupling $(X,X')$ of the uniform distributions on $\mathcal{S}$ and $\mathcal{S}'$, such that
\begin{equation}
    \rho(X,X')\leq 4n.
\end{equation}

It can be checked that $(X,X')$ gives a valid coupling of $K_{\beta,\sigma}^1$ and $K_{\beta,\sigma'}^1$. Moreover,
\begin{equation}
    \mathbb{E}[\rho(X,X')]\leq 4n \mathbb{P}( B_j \neq B_j'  \text{ for some }j\in [n])\leq 16\beta n.
\end{equation}
Therefore by Theorem \ref{Path}, we have that for any $l\geq 0$,
\begin{equation}
    \max_{\sigma_0\in S_n}\|K_{\beta,\sigma_0}^l-\mathbb{P}_{\beta}\|  \leq 2n^2 (16 \beta n)^l.
\end{equation}

\end{proof}

\subsection{Mixing time lower bounds for the Metropolis algorithms}\label{Sect.3.4}

In this subsection, we compare the mixing times of the hit and run algorithms and the Metropolis algorithms for sampling from Mallows permutation models with $L^1$ and $L^2$ distances. 

We consider the following Metropolis algorithm for sampling from $\tilde{\mathbb{P}}_{\beta}$. Each step of the Metropolis algorithm proceeds as follows. Suppose we start at $\sigma$. Independently pick two numbers $i,j\in [n]$ uniformly at random; if $i=j$, stay at $\sigma$; if $i\neq j$, move to $\sigma (i,j)$ with probability 
\begin{equation*}
    \min\Big\{1,\exp\Big(-2\beta(i-j)(\sigma(i)-\sigma(j))\Big)\Big\},
\end{equation*}
and stay at $\sigma$ with probability
\begin{equation*}
   1- \min\Big\{1,\exp\Big(-2\beta(i-j)(\sigma(i)-\sigma(j))\Big)\Big\}.
\end{equation*}
It can be checked that the stationary distribution of the algorithm is $\tilde{\mathbb{P}}_{\beta}$. We denote by $\tilde{M}^l_{\beta,\sigma_0}$ the distribution of the chain after $l$ steps from the starting state $\sigma_0$.

Theorem \ref{Metropolis_1} below gives an order $n\log{n}$ mixing time lower bound for the Metropolis algorithm started at the identity permutation. It is proved by keeping track of the number of fixed points. Comparing Theorem \ref{rho} and Theorem \ref{Metropolis_1}, we see that the hit and run algorithm mixes much faster than the Metropolis algorithm.

\begin{theorem}\label{Metropolis_1}
Suppose that $\beta=\theta\slash n^2$ with $\theta>0$ independent of $n$. There exist positive constants $C(\theta),N(\theta)$ that only depend on $\theta$, such that if $n\geq N(\theta)$, then the $1\slash 4$-mixing time of the Metropolis algorithm for sampling from $\tilde{\mathbb{P}}_{\beta}$ satisfies
\begin{equation}
    t_{mix}(\frac{1}{4})\geq C(\theta) n\log{n}.
\end{equation}
\end{theorem}

\begin{proof}
For any $\sigma\in S_n$, let $F(\sigma)=\sum\limits_{i=1}^n \mathbbm{1}_{\sigma(i)=i}$ be the number of fixed points of $\sigma$. By \cite[Corollary 1.12]{M2}, there exists a constant $C_1(\theta)\in\mathbb{N}^{*}$ that only depends on $\theta$, such that
\begin{equation}\label{eqnnn2}
    \tilde{\mathbb{P}}_{\beta}(F(\sigma)\geq C_1(\theta))\leq \frac{1}{8}.
\end{equation}

For any $l\in\mathbb{N}_{+}$, let $I_{2l-1},I_{2l}$ be the two numbers picked at the $l$th step. Then for any $l\in\mathbb{N}_{+}$, $I_1,\cdots,I_{2l}$ are i.i.d. uniform random variables on $[n]$. Let $\mathcal{R}_t$ be the set formed by $I_1,\cdots,I_{t}$ for any $t\in\mathbb{N}_{+}$, and let $\sigma_l$ be the state of the Metropolis algorithm at step $l\in\mathbb{N}_{+}$ started at the identity permutation $Id$. Then
\begin{equation*}
    \{|\mathcal{R}_{2l}|\leq n-C_1(\theta)\} \subseteq \{F(\sigma_l)\geq C_1(\theta) \}.
\end{equation*}
Hence
\begin{equation}\label{eqnnn1}
    \mathbb{P}(F(\sigma_l)\geq C_1(\theta))\geq \mathbb{P}(|\mathcal{R}_{2l}|\leq n-C_1(\theta)).
\end{equation}

Consider $n\geq C_1(\theta)+1$. For every $k\in [n]$, let $\tau_k=\inf\{l:|\mathcal{R}_l|=k\}$. We also set $\tau_0=0$, and let $X_k=\tau_k-\tau_{k-1}$ for any $k\in [n]$. Note that $X_1,\cdots,X_n$ are independent random variables with $X_k$ following the geometric distribution with success probability $q_k=1-\frac{k-1}{n}$. Hence $\mathbb{E}[X_k]=q_k^{-1}$ and $Var(X_k)\leq q_k^{-2}$. Thus we have
\begin{equation*}
    \mathbb{E}[\tau_{n-C_1(\theta)}]=\sum_{k=1}^{n-C_1(\theta)}\mathbb{E}[X_k]=n\sum_{k=C_1(\theta)+1}^{n} k^{-1}\geq n\log(\frac{n}{C_1(\theta)+1}),
\end{equation*}
\begin{equation*}
    Var(\tau_{n-C_1(\theta)})=\sum_{k=1}^{n-C_1(\theta)}Var(X_k)\leq n^2\sum_{k=1}^{\infty}k^{-2}\leq 2n^2.
\end{equation*}
Hence by Chebyshev's inequality,
\begin{eqnarray*}
  &&  \mathbb{P}(\tau_{n-C_1(\theta)}\leq n(\log(\frac{n}{C_1(\theta)+1})-4))\\
  &\leq& \mathbb{P}(|\tau_{n-C_1(\theta)}-\mathbb{E}[\tau_{n-C_1(\theta)}]|\geq 4n)
     \leq\frac{Var(\tau_{n-C_1(\theta)})}{16 n^2}\leq \frac{1}{8}.
\end{eqnarray*}
Note that for any $l\in\mathbb{N}_{+}$, $\{|\mathcal{R}_{2l}|>n-C_1(\theta)\}\subseteq \{\tau_{n-C_1(\theta)}\leq 2l\}$. Hence by (\ref{eqnnn1}), for any $l\in\mathbb{N}_{+}$ with $l\leq \frac{1}{2}n(\log(\frac{n}{C_1(\theta)+1})-4)$, 
\begin{equation}\label{eqnnn3}
    \tilde{M}^l_{\beta,Id}(F(\sigma)\geq C_1(\theta))=\mathbb{P}(F(\sigma_l)\geq C_1(\theta))\geq \frac{7}{8}.
\end{equation}
Combining (\ref{eqnnn2}) and (\ref{eqnnn3}) gives $\|\tilde{M}^l_{\beta,Id}-\tilde{P}_{\beta}\|> \frac{1}{4}$ for any $l\in\mathbb{N}_{+}$ with $l\leq \frac{1}{2}n(\log(\frac{n}{C_1(\theta)+1})-4)$. Hence there exist positive constants $C(\theta),N(\theta)$ that only depend on $\theta$, such that if $n\geq N(\theta)$, then $t_{mix}(\frac{1}{4})\geq C(\theta)n\log{n}$.
\end{proof}

Similarly, we can consider the Metropolis algorithm for sampling from $\mathbb{P}_{\beta}$. Each step of the Metropolis algorithm is given as follows. Suppose we start at $\sigma$. Independently pick two numbers $i,j\in [n]$ uniformly at random; if $i=j$, stay at $\sigma$; if $i\neq j$, move to $\sigma (i,j)$ with probability
\begin{equation*}
    \min\Big\{1,\exp\Big(-\beta(|\sigma(i)-j|+|\sigma(j)-i|-|\sigma(i)-i|-|\sigma(j)-j|)\Big)\Big\},
\end{equation*} 
otherwise stay at $\sigma$. The following parallel result can be proved using the same argument as in the proof of Theorem \ref{Metropolis_1}. Comparing Theorem \ref{footru} and Theorem \ref{Metropolis_2}, we conclude that the hit and run algorithm mixes much faster than the Metropolis algorithm.

\begin{theorem}\label{Metropolis_2}
Suppose that $\beta=\theta\slash n$ with $\theta>0$ independent of $n$. There exist positive constants $C(\theta),N(\theta)$ that only depend on $\theta$, such that if $n\geq N(\theta)$, then the $1\slash 4$-mixing time of the Metropolis algorithm for sampling from $\mathbb{P}_{\beta}$ satisfies
\begin{equation}
    t_{mix}(\frac{1}{4})\geq C(\theta) n\log{n}
\end{equation}
\end{theorem}

\section{Extension to a two-parameter permutation model}\label{Sect.4}

As explained in the Introduction, there is only one distance metric that is involved in Mallows permutation model. It is natural to consider more general permutation models where multiple distance metrics can be simultaneously incorporated. In Mallows' original paper on Mallows permutation model (\cite{Mal}), a two-parameter model involving a weighted sum of the $L^2$ distance and Kendall's $\tau$ was considered. 

In this section, we consider a new two-parameter permutation model based on the $L^1$ distance and Cayley distance. We introduce an extension of the hit and run algorithm for sampling from this model. The algorithm is a twisted version of the algorithm given in Section \ref{Sect.2.2}. We analyze the mixing time of the twisted algorithm in Section \ref{Sect.4.2}. 

\subsection{The two-parameter permutation model and the twisted algorithm}\label{Sect.4.1}

The two-parameter permutation model we consider here is the following non-uniform probability measure on the symmetric group $S_n$:
\begin{equation}\label{T2}
    \mathbb{P}_{\beta_1,\beta_2}(\sigma)\propto e^{-\beta_1 H(\sigma,Id)+\beta_2 C(\sigma)},
\end{equation}
where $\beta_1,\beta_2$ are parameters and $\sigma\in S_n$. Here, we recall that $H(\sigma,\tau)$ is the $L^1$ distance as defined in (\ref{Dis}), and $C(\sigma)$ is the number of cycles of $\sigma$ for any $\sigma\in S_n$. Note that Cayley distance between $\sigma$ and $Id$ can be expressed as $n-C(\sigma)$, so equivalently speaking the above model involves the $L^1$ distance and Cayley distance. The model can be viewed as ``twisting'' Mallows permutation model with $L^1$ distance by the factor $e^{\beta_2 C(\sigma)}$. We note that the idea of twisting via Cayley distance also appeared in T\'{o}th's work \cite{Toth} that relates the spin $1\slash 2$ Heisenberg ferromagnet to the interchange process. 

In this subsection, we introduce an extension of the hit and run algorithm for sampling from $\mathbb{P}_{\beta_1,\beta_2}$ when $\beta_1>0$. The algorithm involves ``twisting'' the second part of the algorithm for sampling from $\mathbb{P}_{\beta_1}$ in Section \ref{Sect.2.2}. The twisted algorithm is a Markov chain on $S_n$, each step of which consists of two sequential parts as given below. Note that when $\beta_2=0$, the algorithm specializes to the one given in Section \ref{Sect.2.2}.

\begin{itemize}
    \item Starting from $\sigma$, for each $i\in [n]$, independently sample $u_i$ from the uniform distribution on $[0, e^{-2\beta_1 (\sigma(i)-i)_{+} }]$. Let $b_i=i-\frac{\log(u_i)}{2\beta_1}$, and $\mathbf{b}=(b_1,\cdots,b_n)$.
    \item Sample $\sigma'\in S_n$ from the following probability distribution on $S_n$:
    \begin{equation*}
        \mathbb{Q}_{\mathbf{b},\beta_2}(\tau)\propto e^{\beta_2 C(\tau)}1_{\tau(i)\leq b_i\text{ for every }i\in [n]}, \text{ for every }\tau\in S_n. 
    \end{equation*}
    Then move to the new state $\sigma'$.
\end{itemize}

Now we explain how to sample from $\mathbb{Q}_{\mathbf{b},\beta_2}$ when $b_i\geq i$ for every $i\in [n]$. Note that this assumption is satisfied when $\mathbf{b}$ is generated by the first step of the twisted algorithm. Let $N_i=\#\{j\in [n]: b_j\geq i\}-n+i$ for every $i\in [n]$.

In the procedure below for sampling $\tau$ from $\mathbb{Q}_{\mathbf{b},\beta_2}$, we sequentially determine $\tau^{-1}(l)$ for $l=n,n-1,\cdots,1$. For any given $l\in [n]$, suppose $\tau^{-1}(n),\cdots,\tau^{-1}(l)$ have been determined. We call a sequence $(a_1,\cdots,a_s)$ with $s\geq 2$ an ``open arc'' if we have $a_1,\cdots,a_{s-1}\in\{l,\cdots,n\}$, $a_s\in\{1,\cdots,l-1\}$,  $\tau^{-1}(a_w)=a_{w+1}$ for every $w\in\{1,\cdots, s-1\}$, and that there is no $q\in\{l,\cdots,n\}$ such that $\tau^{-1}(q)=a_1$. We call $a_1$ the ``head'' of the open arc, and $a_s$ the ``tail'' of the open arc. We also call $(a_1)$ an open arc if $a_1\in\{1,\cdots,l-1\}$ and there is no $q\in\{l,\cdots,n\}$ such that $\tau^{-1}(q)=a_1$. In this case, $a_1$ is both the head and the tai of this open arc. Intuitively, an open arc is an incomplete cycle formed in the middle of the sampling procedure. 

Now we describe the procedure for sampling $\tau$ from $\mathbb{Q}_{\mathbf{b},\beta_2}$. First note that by our assumption, $b_n\geq n$, hence $N_n\geq 1$. We look at the $N_n$ places of $i$ where $b_i\geq n$, place the symbol $n$ at the place $n$ with probability $\frac{e^{\beta_2}}{e^{\beta_2}+N_n-1}$, and place the symbol $n$ at each of the remaining $N_n-1$ places with probability $\frac{1}{e^{\beta_2}+N_n-1}$. 

Then we look at the remaining $N_{n-1}$ places where $b_i\geq n-1$ (with the place where $n$ was placed deleted). Suppose that currently the open arc containing $n-1$ has $a$ as its head. Note that $b_a\geq a\geq n-1$. We place the symbol $n-1$ at the place $a$ with probability $\frac{e^{\beta_2}}{e^{\beta_2}+N_{n-1}-1}$, and place the symbol $n-1$ at each of the remaining $N_{n-1}-1$ places with probability $\frac{1}{e^{\beta_2}+N_{n-1}-1}$. 

The remaining steps can be similarly defined as above.

By the above procedure, it can be checked that the probability of obtaining $\tau\in S_n$ is
\begin{equation}
    \frac{e^{\beta_2 C(\tau)}}{\prod_{l=1}^n(e^{\beta_2}+N_{l}-1)}1_{\tau(i)\leq b_i\text{ for every } i\in [n] } \propto  e^{\beta_2 C(\tau)} 1_{\tau(i)\leq b_i\text{ for every } i\in [n] }.
\end{equation}
Thus the above sampling procedure indeed generates $\tau$ from $\mathbb{Q}_{\mathbf{b},\beta}$.

The following proposition validates the correctness of the twisted algorithm.

\begin{proposition}
We denote by $K_{\beta_1,\beta_2}$ the transition matrix of the twisted algorithm. Then $K_{\beta_1,\beta_2}$ is reversible with respect to $\mathbb{P}_{\beta_1,\beta_2}$.
\end{proposition}
\begin{proof}
We note that
\begin{equation}
    \mathbb{P}_{\beta_1,\beta_2}(\sigma)=Z_{\beta_1,\beta_2}^{-1}e^{-2\beta_1\sum_{i=1}^n(\sigma(i)-i)_{+}+\beta_2 C(\sigma)},
\end{equation}
where $\sigma\in S_n$ and $Z_{\beta_1,\beta_2}$ is the normalizing constant.

First we note that $u_i\leq 1$ for every $i\in [n]$. In this case, $\tau(i)\leq b_i$ is equivalent to $u_i\leq e^{-2\beta_1 (\tau(i)-i)_{+}}$ for any $\tau  \in S_n$.
Note that for any $\sigma,\sigma'\in S_n$, we have
\begin{eqnarray*}
&&K_{\beta_1,\beta_2}(\sigma,\sigma')\\
&=&\int_{\prod_{i=1}^n[0,e^{-2\beta_1 (\sigma(i)-i)_{+}}]} \frac{e^{\beta_2 C(\sigma')} 1_{\sigma'(i)\leq b_i, \forall  i\in [n]}}{\sum_{\tau\in S_n} e^{\beta_2 C(\tau)} 1_{\tau(i)\leq b_i, \forall i\in [n]}}d u_1 \cdots d u_n\\
&\times& e^{2\beta_1 \sum_{i=1}^n (\sigma(i)-i)_{+}}\\
&=& \int_{\prod_{i=1}^n[0,e^{-2\beta_1 (\sigma(i)-i)_{+}}]}  \frac{e^{\beta_2 C(\sigma')}1_{u_i\leq e^{-2\beta_1 (\sigma'(i)-i)_{+}}, \forall i\in [n]}}{\sum_{\tau\in S_n} e^{\beta_2 C(\tau)} 1_{u_i\leq e^{-2\beta_1 (\tau(i)-i)_{+}}, \forall i\in [n]}}d u_1 \cdots d u_n\\
&\times&   e^{2\beta_1 \sum_{i=1}^n (\sigma(i)-i)_{+}}\\
&=& \int_{B(\sigma,\sigma')} \frac{e^{\beta_2 C(\sigma')}}{\sum_{\tau\in S_n} e^{\beta_2 C(\tau)} 1_{u_i\leq e^{-2\beta_1 (\tau(i)-i)_{+}}, \forall i\in [n]}} d u_1 \cdots d u_n\\
&\times& e^{2\beta_1 \sum_{i=1}^n (\sigma(i)-i)_{+}},
\end{eqnarray*}
where $B(\sigma,\sigma')=\prod_{i=1}^n[0,e^{-2\beta_1 \max\{(\sigma(i)-i)_{+},(\sigma'(i)-i)_{+}\}}]$

Hence
\begin{eqnarray*}
&&\mathbb{P}_{\beta_1,\beta_2}(\sigma)K_{\beta_1,\beta_2}(\sigma,\sigma')\\
&=& Z_{\beta_1,\beta_2}^{-1}\int_{B(\sigma,\sigma')} \frac{e^{\beta_2(C(\sigma)+C(\sigma'))}}{\sum_{\tau\in S_n} e^{\beta_2 C(\tau)} 1_{u_i\leq e^{-2\beta_1 (\tau(i)-i)_{+}}, \forall i\in [n]}} d u_1 \cdots d u_n.
\end{eqnarray*}
Similarly,
\begin{eqnarray*}
&&\mathbb{P}_{\beta_1,\beta_2}(\sigma')K_{\beta_1,\beta_2}(\sigma',\sigma)\\
&=& Z_{\beta_1,\beta_2}^{-1}\int_{B(\sigma,\sigma')} \frac{e^{\beta_2(C(\sigma)+C(\sigma'))}}{\sum_{\tau\in S_n} e^{\beta_2 C(\tau)} 1_{u_i\leq e^{-2\beta_1 (\tau(i)-i)_{+}}, \forall i\in [n]}} d u_1 \cdots d u_n.
\end{eqnarray*}
Therefore
\begin{equation}
    \mathbb{P}_{\beta_1,\beta_2}(\sigma)K_{\beta_1,\beta_2}(\sigma,\sigma')=\mathbb{P}_{\beta_1,\beta_2}(\sigma')K_{\beta_1,\beta_2}(\sigma',\sigma).
\end{equation}
Hence $K_{\beta_1,\beta_2}$ is reversible with respect to $\mathbb{P}_{\beta_1,\beta_2}$.

\end{proof}

We display here some visualization of the two-parameter permutation model (\ref{T2}) based on the twisted algorithm. Figure \ref{figd} shows histograms of the number of fixed points and the length of the cycle containing $500$ when $n=1000$ for $\beta_1=\frac{1}{n}$ and various choices of the parameter $\beta_2$ for the two-parameter permutation model. We also include the corresponding histograms for the $L^1$ model with $\beta=\frac{1}{n}$ (which corresponds to the two-parameter permutation model with $\beta_1=\frac{1}{n}$ and $\beta_2=0$). We observe that when $\beta_1=\frac{1}{n}$ and $\beta_2\in\{0.5,1\}$, the number of fixed points is approximately Poisson with parameter varying with $\beta_2$ and greater than that of the $L^1$ model with $\beta=\frac{1}{n}$. Moreover, we observe that the distribution of the length of the cycle containing $500$ under the two-parameter permutation model can be significantly different from that under the $L^1$ model. We also present some scatter plots of the two-parameter permutation model in Figure \ref{fige}. We observe as we increase $\beta_2\geq 0$ with $\beta_1=\frac{1}{n}$ fixed, the points $\{(i,\sigma(i))\}_{i=1}^n$ for $\sigma$ drawn from the corresponding model are more concentrated around the diagonal $\{(i,i):i\in [n]\}$. When $\beta_2$ is large (say $\beta_2=10$ as in Subfigure (a4) of Figure \ref{fige}), most points $(i,\sigma(i))$ are concentrated around the diagonal. Based on this observation, we focus on the case when $\beta_2\geq 0$ is of constant order for the mixing time analysis in Section \ref{Sect.4.2}.

\begin{figure}[H]
  \centering
  \includegraphics[
  height=\textwidth,
  width=\textwidth
  ]{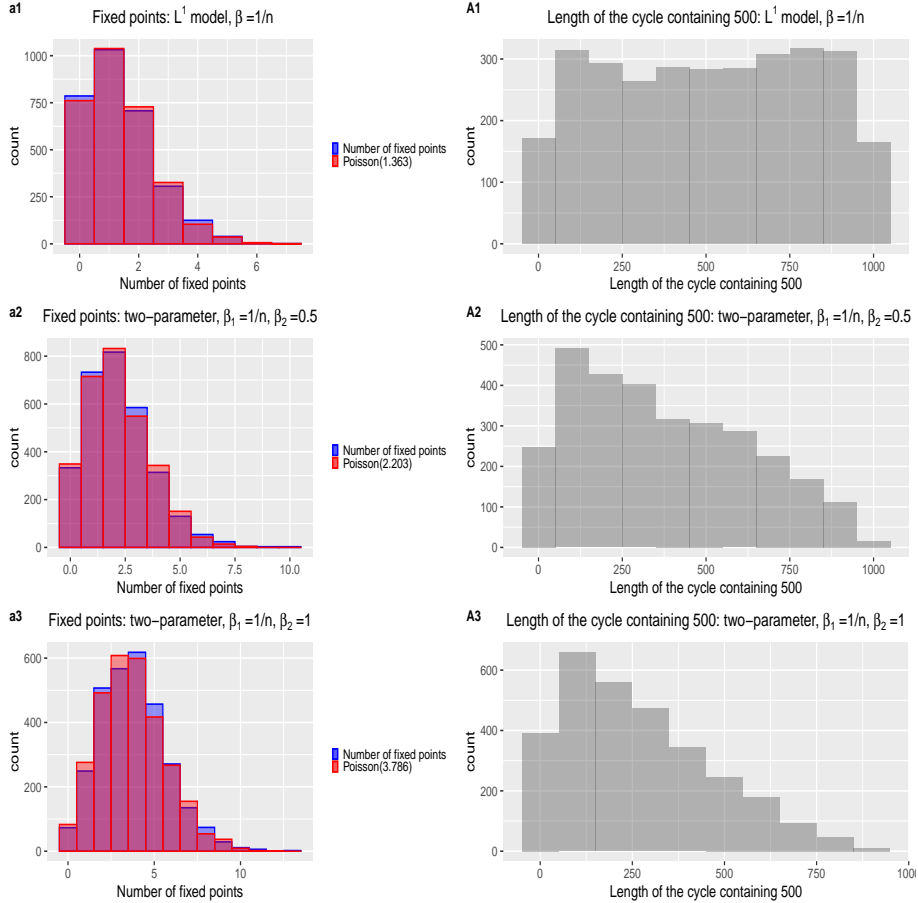}
 \caption{Histograms of the number of fixed points and the length of the cycle containing $500$ of the two-parameter permutation model with $n=1000$ and various choices of the parameter $\beta_2$. The histograms are based on the first $4000$ samples (dropping the first $1000$ samples as burn-in) of the twisted algorithm (introduced in this section) started at the identity permutation. We also present the histograms of $3000$ independent Poisson random variables (with varying parameters) in red for comparison. In Subfigures (a1) and (A1) we also present the corresponding histograms of the $L^1$ model with $\beta=\frac{1}{n}$ for comparison.}
 \label{figd}
\end{figure} 

\begin{figure}[H]
  \centering
  \includegraphics[
  height=\textwidth,
  width=\textwidth
  ]{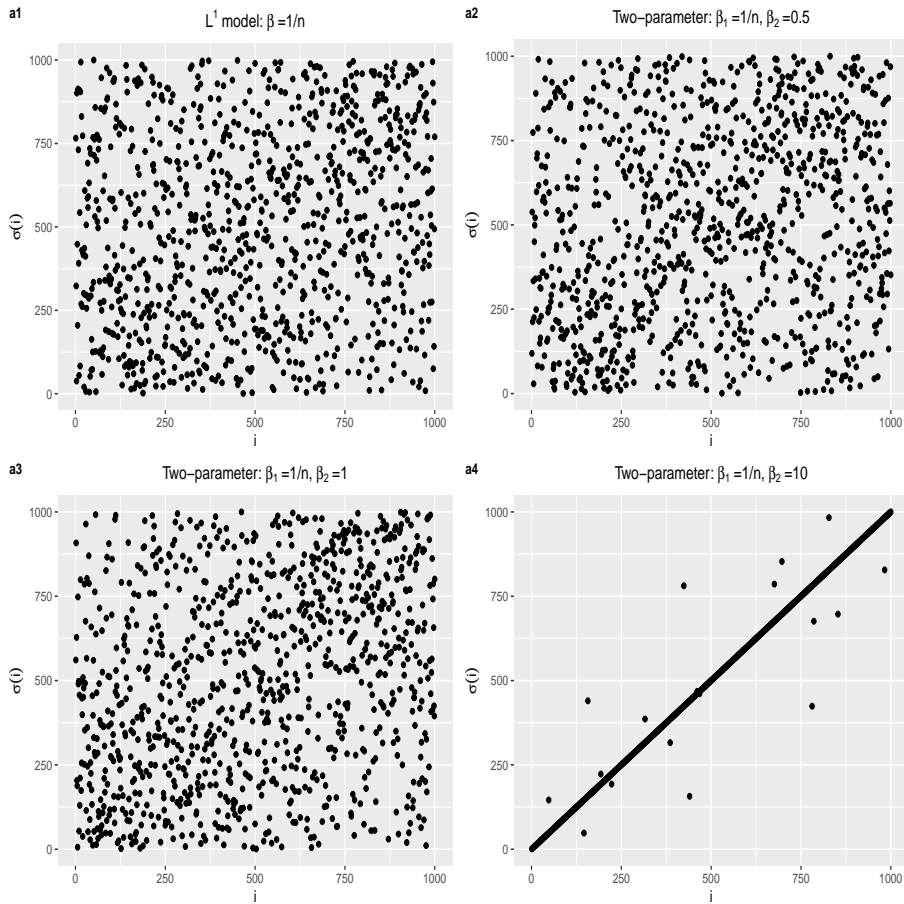}
 \caption{Scatter plots of the two-parameter permutation model with $n=1000$ and various choices of the parameter $\beta_2$. The scatter plots are based on the sample collected at the $4000$th step of the twisted algorithm (introduced in this section) started at the identity permutation. In Subfigure (a1) we also present the scatter plot of the $L^1$ model with $\beta=\frac{1}{n}$ for comparison.}
 \label{fige}
\end{figure} 

At present writing, we cannot see how to adapt the twisted algorithm to Mallows' two-parameter model (\cite{Mal}). 

\subsection{Mixing time analysis of the twisted algorithm}\label{Sect.4.2}

In this subsection, we use path coupling to analyze the mixing time of the twisted algorithm. Recall that the transition matrix of the twisted algorithm is denoted by $K_{\beta_1,\beta_2}$. For any $\sigma_0\in S_n$, we let $K_{\beta_1,\beta_2,\sigma_0}^l$ be the distribution of the chain after $l$ steps from the starting state $\sigma_0$.

The main result is the following theorem which gives a mixing time upper bound of order $\log{n}$.

\begin{theorem}\label{ewens}
There exist absolute constants $c_1^{*},c_2^{*}>0$ and $N\in\mathbb{N}_{+}$ such that the following holds. Suppose that $0<\beta_1\leq \frac{c_1^{*}}{n}$, $0\leq \beta_2\leq c_2^{*}$ and $n\geq N$. Then for any $l\geq 0$,
\begin{equation}
    \max_{\sigma_0\in S_n}\|K_{\beta_1,\beta_2,\sigma_0}^l-\mathbb{P}_{\beta_1,\beta_2}\|\leq n^2(\frac{1}{2})^{l-1}.
\end{equation}
This implies that
\begin{equation}
    t_{mix}(\frac{1}{4})\leq 3+\Big\lceil\frac{2\log{n}}{\log{2}}\Big\rceil.
\end{equation}
\end{theorem}
\begin{remark}
Quantitative bounds on $c_1^{*},c_2^{*},N$ that suffice for the conclusion of Theorem \ref{ewens} to hold can be easily obtained from the proof of Theorem \ref{ewens}.
\end{remark}

The rest of this subsection is devoted to the proof of Theorem \ref{ewens}.

We have the following lemma on the measure $\mathbb{Q}_{\mathbf{b},\beta_2}$.

\begin{lemma}\label{L3}
Suppose that $\beta_2\geq 0$, $k\in [n]$, $b_k'\in [n]$, and $b_i\in [n]$ for every $i\in [n]$. Further suppose that $b_k'\geq k$ and $b_i\geq i$ for every $i\in [n]$. Denote by $\mathbf{b}=(b_1,\cdots,b_n),\mathbf{b}'=(b_1,\cdots,b_{k-1},b_k',b_{k+1},\cdots,b_n)$. Let $b_i'=b_i$ for every $i\in [n]\backslash \{k\}$, and
\begin{equation*}
    N_l=\# \{i\in[n]: b_i\geq l\}-(n-l),
\end{equation*}
\begin{equation*}
    N_l'=\# \{i\in[n]: b_i'\geq l\}-(n-l),
\end{equation*}
for any $l\in [n]$.

Define 
\begin{equation*}
    \mathcal{S}:=\{\sigma\in S_n: \sigma(i)\leq b_i\text{ for every }i\in[n]\},
\end{equation*}
\begin{equation*}
    \mathcal{S}':=\{\sigma\in S_n: \sigma(i)\leq b_i'\text{ for every }i\in [n]\}.
\end{equation*}
Assume that $\mathcal{S},\mathcal{S}'\neq \emptyset$. Further let $\mu=\mathbb{Q}_{\mathbf{b},\beta_2}$ and $\mu'=\mathbb{Q}_{\mathbf{b}',\beta_2}$. Then there exists a coupling $(X,X')$ of $\mu$ and $\mu'$, such that 
\begin{equation}
    \mathbb{E}[\rho(X,X')] \leq 4 e^{\beta_2}\sum_{l=1}^n\frac{l}{\min\{N_l,N_l'\}}+2n,
\end{equation}
where the path metric $\rho=\rho(u,v)$ is defined in (\ref{Pm}) with the graph $G=(\Omega,E)$ and the length function $l(u,v)$ defined as in Section \ref{Sect.3.2}.
\end{lemma}
\begin{proof}
The proof consists of two parts. In the first part, we construct the coupling $(X,X')$. The coupling is a ``twisted'' version of the coupling in Lemma \ref{Lem}. Then we obtain the desired upper bound on $\mathbb{E}[\rho(X,X')]$ in the second part.

\paragraph{Part 1}

Without loss of generality we assume that $b_k'>b_k$ (note that when $b_k'=b_k$, we have $\mu=\mu'$, and the conclusion holds trivially). Note that in this case, $N_l\leq N_l'$ for every $l\in [n]$.

We construct the coupling as follows. Sequentially for $l=n,\cdots,b_k'+1$, we look at places $i$ with $b_i\geq l$ that have not yet been taken, and let $X^{-1}(l)=X'^{-1}(l)$ following the first $n-b_k'$ steps of the sampling procedure for $\mathbb{Q}_{\mathbf{b},\beta_2}$. Namely, suppose that currently the open arc containing $l$ has $a$ as it head, then we let $X^{-1}(l)=a$ with probability $\frac{e^{\beta_2}}{e^{\beta_2}+N_l-1}$. We also let $X^{-1}(l)$ be each one of the remaining $N_l-1$ available places with probability $\frac{1}{e^{\beta_2}+N_l-1}$.

In the following, we associate a ``mark'' to every number in $[n]$. Originally the number $i$ is marked $i$ for every $i \in [n]$. The mark of a number may change during the evolution.

When $l=b_k'$, there are $N_l$ possible choices of $X^{-1}(l)$, which we denoted by $j_{l,1},\cdots,j_{l,N_l}$, and $N_l+1$ possible choices $j_{l,1},\cdots,j_{l,N_l},k$ of $X'^{-1}(l)$. We define the coupling at this step as follows. As $b_k'>b_k\geq k$, $k$ cannot be the current head of the open arc containing $l=b_k'$ in $X'$. Suppose that $a$ is the current head of the open arc containing $l$ in $X$ (and also in $X'$). For every $s\in\{1,\cdots,N_l\}$, we let $X^{-1}(l)=X'^{-1}(l)=j_{l,s}$ with probability $\frac{1}{e^{\beta_2}+N_l}$ (if $j_{l,s}\neq a$) or $\frac{e^{\beta_2}}{e^{\beta_2}+N_l}$ (if $j_{l,s}=a$). For every $s\in\{1,\cdots,N_l\}$, with probability $\frac{1}{e^{\beta_2}+N_l-1}-\frac{1}{e^{\beta_2}+N_l}$ (if $j_{l,s}\neq a$) or $\frac{e^{\beta_2}}{e^{\beta_2}+N_l-1}-\frac{e^{\beta_2}}{e^{\beta_2}+N_l}$ (if $j_{l,s}=a$), we let $X'^{-1}(l)=k$ and $X^{-1}(l) = j_{l,s}$. Note that for either case, if $l-1>b_k$, the set of possible choices of $X'^{-1}(l-1)$ is the union of the set of possible choices of $X^{-1}(l-1)$ and another element (one of $j_{l,1},\cdots,j_{l,N_l},k$).

Now we define the coupling for $l=b_k'-1,\cdots,b_k+1$ inductively. Assume that the set of possible $N_l+1$ choices of $X'^{-1}(l)$ is the union of the set of marks associated to the possible $N_l$ choices of $X^{-1}(l)$ (these marks are denoted by $j_{l,1},\cdots,j_{l,N_l}$) and another element (denoted by $t_l$). Note that this is satisfied for $l=b_k'-1$. Suppose that the heads of the current open arcs containing $l$ in $X$ and $X'$ are $a$ and $a'$, respectively. 

There are two possible cases: (a) $t_l=a'$; (b) $t_l\neq a'$. First consider \textbf{case (a)}. For every $s\in\{1,\cdots,N_l\}$, we let $X'^{-1}(l)=j_{l,s}$ and $X^{-1}(l)$ be the number that is marked $j_{l,s}$ with probability $\frac{1}{e^{\beta_2}+N_l}$. For every $s\in\{1,\cdots,N_l\}$, with probability $\frac{e^{\beta_2}}{e^{\beta_2}+N_l-1}-\frac{1}{e^{\beta_2}+N_l}$ (if the number that is marked $j_{l,s}$ is $a$) or $\frac{1}{e^{\beta_2}+N_l-1}-\frac{1}{e^{\beta_2}+N_l}$ (if the number that is marked $j_{l,s}$ is not $a$), we let $X'^{-1}(l)=t_l$ and $X^{-1}(l)$ be the number that is marked $j_{l,s}$. 

For \textbf{case (b)}, suppose that $a'=j_{l,s_1}$, $a$ is marked with $j_{l,s_0}$ and that $b$ is marked with $j_{l,s_1}$. There are two sub-cases: (i) $s_0=s_1$; (ii) $s_0\neq s_1$. First consider \textbf{sub-case (i)}. For every $s\in\{1,\cdots,N_l\}$, we let $X'^{-1}(l)=j_{l,s}$ and $X^{-1}(l)$ be the number that is marked $j_{l,s}$ with probability $\frac{e^{\beta_2}}{e^{\beta_2}+N_l}$ (if $s= s_0$) or $\frac{1}{e^{\beta_2}+N_l}$ (if $s\neq s_0$). For every $s\in\{1,\cdots,N_l\}$, we let $X'^{-1}(l)=t_l$ and $X^{-1}(l)$ be the number that is marked $j_{l,s}$ with probability $\frac{e^{\beta_2}}{e^{\beta_2}+N_l-1}-\frac{e^{\beta_2}}{e^{\beta_2}+N_l}$ (if $s= s_0$) or $\frac{1}{e^{\beta_2}+N_l-1}-\frac{1}{e^{\beta_2}+N_l}$ (if $s\neq s_0$). Now we consider \textbf{sub-case (ii)}. For every $s\neq s_0,s_1$, we let $X'^{-1}(l)=j_{l,s}$ and $X^{-1}(l)$ be the number that is marked $j_{l,s}$ with probability $\frac{1}{e^{\beta_2}+N_l}$. With probability $\frac{e^{\beta_2}}{e^{\beta_2}+N_l}$, we let $X'^{-1}(l)=j_{l,s_1}=a'$ and $X^{-1}(l)=a$ (which is marked with $j_{l,s_0}$), and change the mark of $b$ to $j_{l,s_0}$; with probability $\frac{1}{e^{\beta_2}+N_l}$, we let $X'^{-1}(l)=j_{l,s_0}$ and $X^{-1}(l)=b$ (which is marked with $j_{l,s_1}$), and change the mark of $a$ to $j_{l,s_1}$. For every $s\in\{1,\cdots,N_l\}$, we let $X'^{-1}(l)=t_l$ and $X^{-1}(l)$ be the number that is marked $j_{l,s}$ with probability $\frac{e^{\beta_2}}{e^{\beta_2}+N_l-1}-\frac{e^{\beta_2}}{e^{\beta_2}+N_l}$ (if $s= s_0$) or $\frac{1}{e^{\beta_2}+N_l-1}-\frac{1}{e^{\beta_2}+N_l}$ (if $s\neq s_0$). 

It can be checked that for both cases (a)-(b), if $l-1>b_k$, then the assumption that the set of possible $N_{l-1}+1$ choices of $X'^{-1}(l-1)$ is the union of the set of marks associated to the possible $N_{l-1}$ choices of $X^{-1}(l-1)$ and another element holds. Thus the above procedure defines a valid coupling for $l=b'_k-1,\cdots,b_k+1$. 

When $l=b_k$, suppose that the numbers with marks $j_{l,1},\cdots,j_{l,N_l-1},k$ are the $N_l$ possible choices of $X^{-1}(l)$. (Note that when $l=b_k$, the number marked with $k$ is just $k$.) Then $X'^{-1}(l)$ also has $N_l$ possible choices, which are given by $j_{l,1},\cdots,j_{l,N_l-1}$ and some $t_l\notin \{j_{l,1},\cdots,j_{l,N_l-1}\}$. Let $a,a'$ be the heads of the current open arcs containing $l$ of $X,X'$, respectively. If $a=k$, we let $s_0=N_l$; otherwise, suppose that $a$ is marked with $j_{l,s_0}$. If $a'=t_l$, we let $s_1=N_l$; otherwise, suppose that $a'=j_{l,s_1}$. Now there are two cases: (a) $s_0=s_1$; (b) $s_0\neq s_1$. 

For \textbf{case (a)}, there are two sub-cases: (i) $s_0=N_l$; (ii) $s_0\leq N_l-1$. For \textbf{sub-case (i)}, for every $s\in\{1,\cdots,N_l-1\}$, we let $X'^{-1}(l)=j_{l,s}$ and let $X^{-1}(l)$ be the number marked with $j_{l,s}$ with probability $\frac{1}{e^{\beta_2}+N_l-1}$. With probability $\frac{e^{\beta_2}}{e^{\beta_2}+N_l-1}$, we let $X'^{-1}(l)=t_l$ and $X^{-1}(l)=k$. For \textbf{sub-case (ii)}, for every $s\in \{1,\cdots,N_l-1\}\backslash \{s_0\}$, we let $X'^{-1}(l)=j_{l,s}$ and let $X^{-1}(l)$ be the number marked with $j_{l,s}$ with probability $\frac{1}{e^{\beta_2}+N_l-1}$. With probability $\frac{e^{\beta_2}}{e^{\beta_2}+N_l-1}$, we let $X'^{-1}(l)=j_{l,s_0}$ and let $X^{-1}(l)$ be the number marked with $j_{l,s_0}$. With probability $\frac{1}{e^{\beta_2}+N_l-1}$, we let $X'^{-1}(l)=t_l$ and $X^{-1}(l)=k$

For \textbf{case (b)}, there are three sub-cases: (i) $s_0,s_1\leq N_l-1$; (ii) $s_0= N_l,s_1\leq N_l-1$; (iii) $s_1=N_l,s_0\leq N_l-1$. For \textbf{sub-case (i)}, for every $s\leq N_l-1$ with $s\neq s_0,s_1$, we let $X'^{-1}(l)=j_{l,s}$ and let $X^{-1}(l)$ be the number marked with $j_{l,s}$ with probability $\frac{1}{e^{\beta_2}+N_l-1}$. With probability $\frac{e^{\beta_2}}{e^{\beta_2}+N_l-1}$, we let $X^{-1}(l)$ be the number marked with $j_{l,s_0}$ and let $X'^{-1}(l)=j_{l,s_1}$, and change the mark of the number that is originally marked $j_{l,s_1}$ to $j_{l,s_0}$.  With probability $\frac{1}{e^{\beta_2}+N_l-1}$, we let $X^{-1}(l)$ be the number marked with $j_{l,s_1}$ and let $X'^{-1}(l)=j_{l,s_0}$, and change the mark of the number that is originally marked $j_{l,s_0}$ to $j_{l,s_1}$. With probability $\frac{1}{e^{\beta_2}+N_l-1}$, we let $X'^{-1}(l)=t_l$ and $X^{-1}(l)$ be the number marked with $k$. For \textbf{sub-case (ii)}, for every $s\leq N_l-1$ with $s\neq s_1$, we let $X'^{-1}(l)=j_{l,s}$ and let $X^{-1}(l)$ be the number marked with $j_{l,s}$ with probability $\frac{1}{e^{\beta_2}+N_l-1}$. With probability $\frac{e^{\beta_2}}{e^{\beta_2}+N_l-1}$, we let $X^{-1}(l)$ be the number marked with $k$ and let $X'^{-1}(l)=j_{l,s_1}$, and change the mark of the number that is originally marked $j_{l,s_1}$ to $k$. With probability $\frac{1}{e^{\beta_2}+N_l-1}$, we let $X^{-1}(l)$ be the number marked with $j_{l,s_1}$ and let $X'^{-1}(l)=t_l$, and change the mark of the number that is originally marked $k$ to $j_{l,s_1}$. For \textbf{sub-case (iii)}, for every $s\leq N_l-1$ with $s\neq s_0$, we let $X'^{-1}(l)=j_{l,s}$ and let $X^{-1}(l)$ be the number marked with $j_{l,s}$ with probability $\frac{1}{e^{\beta_2}+N_l-1}$. With probability $\frac{e^{\beta_2}}{e^{\beta_2}+N_l-1}$, we let $X^{-1}(l)$ be the number marked with $j_{l,s_0}$ and let $X'^{-1}(l)=t_l$, and change the mark of the number that is originally marked $k$ to $j_{l,s_0}$. With probability $\frac{1}{e^{\beta_2}+N_l-1}$, we let $X^{-1}(l)$ be the number marked with $k$ and let $X'^{-1}(l)=j_{l,s_0}$, and change the mark of the number that is originally marked $j_{l,s_0}$ to $k$.

We define the coupling for $l=b_k-1,\cdots,1$ inductively as below. Assume that either of the following two possibilities is true: (a) the $N_l$ possible choices of $X'^{-1}(l)$ are the same as the marks associated with the $N_l$ possible choices of $X^{-1}(l)$; (b) the $N_l$ possible choices of $X^{-1}(l)$ are the numbers marked with $j_{l,1},\cdots,j_{l,N_l-1},k$ for some $j_{l,1},\cdots,j_{l,N_l-1}$, and the $N_l$ possible choices of $X'^{-1}(l)$ are $j_{l,1},\cdots,j_{l,N_l-1},t_l$ for some $t_l\notin \{j_{l,1},\cdots,j_{l,N_l-1}\}$. Note that the assumption holds when $l=b_k-1$. For both cases, the coupling can be defined in a similar way as the case when $l=b_k$. It can be checked that if $l-1\geq 1$, the assumption holds with $l$ replaced by $l-1$.

This finishes the construction of the coupling. 

\paragraph{Part 2}

In the following, we bound $\mathbb{E}[\rho(X,X')]$.

Consider the set of steps $l\in [n]$ where there is a change of some mark. Suppose the elements of this set are ordered as $l_1>l_2>\cdots>l_r$. Suppose that at step $l_i$, the mark of a number is changed from $a_i$ to $b_i$. Now starting from $X^{-1}$, sequentially for $i=1,\cdots,r$, we transpose the symbols $a_i$ and $b_i$, and obtain a new permutation $X''^{-1}$ in the end. That is,
\begin{equation*}
   X''^{-1}=(a_r,b_r)\cdots(a_1,b_1) X^{-1}.
\end{equation*}
Note that the $i$th transposition can be decomposed into at most $2l_i$ adjacent transpositions. Hence
\begin{equation}
    \rho(X,X'')\leq 2\sum_{i=1}^r l_i=2\sum_{l=1}^n l 1_{B_l},
\end{equation}
where $B_l$ is the event that there is a change of some mark at step $l$. Now note that
\begin{equation}
    \mathbb{P}(B_l)\leq \frac{2e^{\beta_2}}{e^{\beta_2}+N_l-1}\leq \frac{2e^{\beta_2}}{N_l}.
\end{equation}
Thus we have
\begin{equation}\label{Eq1}
    \mathbb{E}[\rho(X,X'')]\leq 4e^{\beta_2}\sum_{l=1}^n l N_l^{-1}.
\end{equation}

Order the elements from the set $\{l\in [n]: X'^{-1}(l)\neq X''^{-1}(l)\}$ as $m_1>m_2>\cdots>m_p$. Then from our construction, we have that $X'^{-1}(m_1)=k,X''^{-1}(m_p)=k$, and $X''^{-1}(m_{s})=X'^{-1}(m_{s+1})$ for any $s=1,\cdots,p-1$. Similar to the last part of the proof of Lemma \ref{Lem}, we have
\begin{equation}
    \rho(X'',X')\leq 2n.
\end{equation}
Hence
\begin{equation}\label{Eq2}
    \mathbb{E}[\rho(X'',X')]\leq 2n.
\end{equation}

Now using (\ref{Eq1}), (\ref{Eq2}) and the fact that $\rho(X,X')\leq \rho(X,X'')+\rho(X'',X')$, we have
\begin{equation}
     \mathbb{E}[\rho(X,X')]\leq 4e^{\beta_2}\sum_{l=1}^n l N_l^{-1}+2n.
\end{equation}

\end{proof}

\begin{proof}[Proof of theorem \ref{ewens}]
We assume that $0<\beta_1\leq \frac{c_1}{n}$ and $0\leq \beta_2   \leq c_2$ for some $c_1,c_2>0$. We use the path coupling technique. The choice of the graph structure $G=(\Omega,E)$ and the path metric $\rho$ is the same as that from the proof of Theorem \ref{rho}. We have $diam(\Omega)\leq 2n^2$.

Consider two permutations $\sigma,\sigma'\in  S_n$ that are adjacent in $G$. Assume that $\sigma'=(i,i+1)\sigma$, where $i\in\{1,\cdots,n-1\}$. Suppose that $\sigma(k_1)=i$ and $\sigma(k_2)=i+1$. Thus $\sigma'(k_1)=i+1$ and $\sigma'(k_2)=i$.
In the following, we construct a coupling $(X,X')$ of the distributions $K_{\beta_1,\beta_2,\sigma}^1$ and $K_{\beta_1,\beta_2,\sigma'}^1$.

For the first step of the hit and run algorithm, we define $B_j$ and $B_j'$ and their coupling the same way as in the proof of Theorem \ref{footru}, taking $\beta=\beta_1$. We recall from the proof of Theorem \ref{footru} that $B_j=B_j'$ for $j\neq k_1,k_2$, and that
\begin{equation}
    \mathbb{P}(B_{k_1}\neq B'_{k_1})=1-e^{-2\beta_1}, \quad \mathbb{P}(B_{k_2}\neq B'_{k_2})=1-e^{-2\beta_1}.
\end{equation}

We let $\mathcal{G}=\sigma(B_1,B_1',\cdots,B_n,B_n')$. We also let $\mathcal{F}=\sigma(B_{k_1},B'_{k_1},B_{k_2},B'_{k_2})$.

Now we consider the second part of the twisted algorithm. We denote by $\mathbf{B}=(B_1,\cdots,B_n)$ and $\mathbf{B}'=(B_1',\cdots,B_n')$. Let $N_l=\#\{i\in [n]: B_i\geq l\}-(n-l)$, $N_l'=\#\{i\in [n]: B_i'\geq l\}-(n-l)$. 

If $B_{k_1}=B_{k_1}'$ and $B_{k_2}=B_{k_2}'$, we let $X=X'$ be sampled from $\mathbb{Q}_{\mathbf{B},\beta_2}$. In this case,
\begin{equation}
    \mathbb{E}[\rho(X,X')|\mathcal{G}]=0.
\end{equation}

Otherwise, either $B_{k_1}\neq B_{k_1}'$ or $B_{k_2}\neq B_{k_2}'$. By Lemma \ref{L3}, similar to the proof of Theorem \ref{rho}, conditional on $\mathcal{G}$, there exists a coupling $(X,X')$ between $\mathbb{Q}_{\mathbf{B},\beta_2} $ and $\mathbb{Q}_{\mathbf{B'},\beta_2}$, such that
\begin{equation}
    \mathbb{E}[\rho(X,X')|\mathcal{G}]\leq 8e^{\beta_2}\sum_{l=1}^n\frac{l}{\min\{N_l,N_l'\}}+4n.
\end{equation}

Now note that as $B_j\geq j$ for every $j\in [n]$. We have
\begin{eqnarray*}
    N_l&=&\#\{j\in [n]: B_j\geq  l\}-(n-l)\\
    &=& 1+\sum_{j=1}^{l-1}1_{B_j\geq l}\\
    &\geq & 1+\sum_{j:1\leq j\leq l-1,j\neq k_1,k_2}1_{B_j\geq l}. 
\end{eqnarray*}
For any $1\leq j\leq l-1$ and any $l\in [n]$, we have
\begin{equation}
    \mathbb{P}(B_j\geq l)\geq \mathbb{P}(B_j=n)\geq e^{-2\beta_1 n}. 
\end{equation}

By Hoeffding's inequality, for any $t\geq 0$ and any $l\geq 4$, we have
\begin{equation}
    \mathbb{P}(\frac{N_l-1}{l-3}-e^{-2\beta_1 n}\leq -t|\mathcal{F})\leq e^{-2(l-3)t^2}.
\end{equation}
Take $t=\frac{1}{2}e^{-2\beta_1 n}$. We obtain that for any $l\geq 4$,
\begin{equation}
    \mathbb{P}(N_l\leq 1+\frac{l-3}{2}e^{-2\beta_1 n}|\mathcal{F})\leq e^{-\frac{1}{2}(l-3)e^{-4\beta_1 n}}.
\end{equation}
Similarly, for any $l\geq 4$,
\begin{equation}
    \mathbb{P}(N_l'\leq 1+\frac{l-3}{2}e^{-2\beta_1 n}|\mathcal{F})\leq e^{-\frac{1}{2}(l-3)e^{-4\beta_1 n}}.
\end{equation}
Hence for any $l\geq \max\{4,n^{\frac{1}{2}}\}$,
\begin{equation}
    \mathbb{P}(\min\{N_l,N_l'\}\leq 1+\frac{l}{8}e^{-2c_1}|\mathcal{F})\leq 2 e^{-\frac{n^{\frac{1}{2}}}{8}e^{-4 c_1}}.
\end{equation}
Note that $N_l\geq 1$, $N_l'\geq 1$. Therefore
\begin{eqnarray*}
\mathbb{E}[\sum_{l=1}^n\frac{l}{\min\{N_l,N_l'\}}|\mathcal{F}] &\leq& \sum_{l=1}^4 l +\sum_{l=1}^{\lfloor n^{\frac{1}{2}}\rfloor}l+8 e^{2 c_1}n+2n^2 e^{-\frac{n^{\frac{1}{2}}}{8}e^{-4c_1}}\\
&\leq& 10+ (1+8 e^{2c_1}) n+2n^2 e^{-\frac{n^{\frac{1}{2}}}{8}e^{-4c_1}}.
\end{eqnarray*}

Let
\begin{equation*}
    f(c_1,c_2,n)=80 e^{c_2}+(8e^{c_2}+64e^{c_2+2c_1}+4)n+16e^{c_2}n^2 e^{-\frac{n^{\frac{1}{2}}}{8}e^{-4c_1}}.
\end{equation*}
We have
\begin{eqnarray*}
 \mathbb{E}[\rho(X,X')|    \mathcal{F}]\leq f(c_1,c_2,n) 1_{B_{k_1}\neq B_{k_1'}\text{ or } B_{k_2}\neq B_{k_2'}}.
\end{eqnarray*}
This leads to
\begin{eqnarray*}
     \mathbb{E}[\rho(X,X')] &\leq& f(c_1,c_2,n)\mathbb{P}(B_{k_1}\neq B_{k_1'}\text{ or } B_{k_2}\neq B_{k_2'})\\
     &\leq& 4\beta_1 f(c_1,c_2,n)\leq g(c_1,c_2,n),
\end{eqnarray*}
where 
\begin{equation*}
    g(c_1,c_2,n)=\frac{320 c_1 e^{c_2}}{n}+4c_1(8e^{c_2}+64e^{c_2+2c_1}+4)+64 c_1 e^{c_2} n  e^{-\frac{n^{\frac{1}{2}}}{8}e^{-4c_1}}.
\end{equation*}
By Theorem \ref{Path}, if $g(c_1,c_2,n)<1$, then for any $l\geq 0$,
\begin{equation}
    \max_{\sigma_0\in S_n}\|K_{\beta_1,\beta_2,\sigma_0}^l-\mathbb{P}_{\beta_1,\beta_2} \| \leq 2n^2 (g(c_1,c_2,n))^l.
\end{equation}
This leads to the conclusion of Theorem \ref{ewens}.

\end{proof}

\section{Extension to lattice permutations in higher dimensions}\label{Sect.5}

In this section, we consider higher dimensional analogs of Mallows permutation models with $L^1$ and $L^2$ distances. These analogs are called ``lattice permutations'' (see e.g. \cite{GLU}). In the mathematical physics literature, these higher dimensional permutation models are of interest due to their connections to Feynman's approach to the study of quantum Bose gas \cite{Fey}. The reader is referred to e.g. \cite{Bet,Fic,GRU,GLU} for some recent developments in the literature. 

We introduce these permutation models as follows. We fix the dimension $d\geq 2$ and a positive integer $N$. Denote by $S_{N,d}$ the set of bijections from $[N]^d$ to itself. For any $\sigma\in S_{N,d}$, we define
\begin{equation}
    H_d(\sigma)=\sum_{\mathbf{x}\in [N]^d}\sum_{i=1}^d |\sigma(\mathbf{x})_i-\mathbf{x}_i|, \quad \tilde{H}_d(\sigma)=\sum_{\mathbf{x}\in [N]^d}\|\sigma(\mathbf{x})-\mathbf{x}\|^2,
\end{equation}
where $\|\cdot\|$ is the Euclidean distance in $\mathbb{R}^d$. For any $\beta>0$, we further define the following probability measures on $S_{N,d}$:
\begin{equation}
    \mathbb{P}_{\beta,d}(\sigma)=Z_{\beta,d}^{-1}\exp(-\beta H_d(\sigma)), \quad \tilde{\mathbb{P}}_{\beta,d}(\sigma)=\tilde{Z}_{\beta,d}^{-1}\exp(-\beta \tilde{H}_d(\sigma)), 
\end{equation}
where $\sigma\in S_{N,d}$, and $Z_{\beta,d}$ and $\tilde{Z}_{\beta,d}$ are the normalizing constants. We call these probability measures ``lattice permutations with $L^1$ and $L^2$ distances'' in the following. Setting $d=1$, these reduce to $\mathbb{P}_{\beta}$ and $\tilde{\mathbb{P}}_{\beta}$--Mallows permutation models with $L^1$ and $L^2$ distances. Using similar arguments as those leading to (\ref{Cor}) and (\ref{foot}), we can derive that for any $\sigma\in S_{N,d}$,
\begin{equation}
    \mathbb{P}_{\beta,d}(\sigma)\propto \exp(-2\beta \sum_{\mathbf{x}\in [N]^d}\sum_{i=1}^d (\sigma(\mathbf{x})_i-\mathbf{x}_i)_{+}),
\end{equation}
\begin{equation}
    \mathbb{P}_{\beta,d}(\sigma)\propto \exp(2\beta \sum_{\mathbf{x}\in [N]^d}\sum_{i=1}^d \mathbf{x}_i  \sigma(\mathbf{x})_i).
\end{equation}

Most of the existing algorithms for simulating from these lattice permutations are based on the Metropolis algorithm (see e.g. \cite{Bet,GRU,GLU}). In this section, we introduce a generalization of the hit and run algorithms introduced in Section \ref{Sect.2} to sample from such higher dimensional models. The generalization combines the hit and run algorithms with a Gibbs sampler.

In the following we introduce the generalized algorithms. We start from $\tilde{\mathbb{P}}_{\beta,d}$--lattice permutations with $L^2$ distance. 

The generalized algorithm is based on a Gibbs sampler for sampling from the uniform distribution on the following subset of $S_{N,d}$:
\begin{equation*}
    \mathcal{M}_{\mathbf{b}}=\{\sigma\in S_{N,d}: (\sigma(\mathbf{x}))_i\geq b_{\mathbf{x},i}\text{ for every }\mathbf{x}\in [N]^d, i\in [d]\},
\end{equation*}
where $\mathbf{b}=(b_{\mathbf{x},i})_{\mathbf{x}\in [N]^d,i\in [d]}$ with $b_{\mathbf{x},i}\in\mathbb{R}$ for every $\mathbf{x}\in [N]^d,i\in [d]$. Assuming that $\mathcal{M}_{\mathbf{b}}\neq \emptyset$, we denote by $\pi_{\mathcal{M}_{\mathbf{b}}}$ the uniform distribution on $\mathcal{M}_{\mathbf{b}}$.

We introduce the Gibbs sampler for sampling from $\pi_{\mathcal{M}_{\mathbf{b}}}$ as follows. Each step of the Gibbs sampler proceeds as follows. Suppose we start at $\sigma\in \mathcal{M}_{\mathbf{b}}$. Let $\sigma_0=\sigma$. Sequentially for $j=1,\cdots,d$, let
\begin{equation*}
    \mathcal{A}_j=\{\tau\in S_{N,d}: (\tau(\mathbf{x}))_l=(\sigma_{j-1}(\mathbf{x}))_l\text{ for every }\mathbf{x}\in [N]^d, l\in [d]\backslash\{j\}\},
\end{equation*}
and sample $\sigma_j\in S_{N,d}$ from $\pi_{\mathcal{M}_{\mathbf{b}}}(\cdot|\mathcal{A}_j)$--the uniform distribution on $\mathcal{M}_{\mathbf{b}}\cap\mathcal{A}_j$. The chain then moves from $\sigma$ to $\sigma_d$. 

The practical implementation of the sampling process for $\sigma_j$ with $j\in [d]$ can be described as follows. For any $\mathbf{y}\in [N]^d$ and $i\in [d]$, we denote by $\mathbf{y}_{-i}=(\mathbf{y}_1,\cdots,\mathbf{y}_{i-1},\mathbf{y}_{i+1},\cdots,\mathbf{y}_d)$. Now for any $\mathbf{z}\in [N]^{d-1}$, we let 
\begin{equation*}
    \mathcal{D}_{\mathbf{z},j}=\{\mathbf{x}\in [N]^d: (\sigma_{j-1}(\mathbf{x}))_{-j}=\mathbf{z}\},
\end{equation*}
and note that $|\mathcal{D}_{\mathbf{z},j}|=N$. Suppose that
\begin{equation*}
    \mathcal{D}_{\mathbf{z},j}=\{\mathbf{w}_{\mathbf{z},j,1},\cdots,\mathbf{w}_{\mathbf{z},j,N}\}.
\end{equation*}
For any $\mathbf{z}\in [N]^{d-1}$, we sample $\gamma_{\mathbf{z},j}\in S_N$ uniformly from the set
\begin{equation*}
    \{\tau\in S_N: \tau(l)\geq b_{\mathbf{w}_{\mathbf{z},j,l},j}\text{ for every }l\in [N]\}.
\end{equation*}
Note that this is a set of permutations with one-sided restrictions, hence $\gamma_{\mathbf{z},j}$ can be sampled via the procedure described in Section \ref{Sect.2}. Then $\sigma_j$ is determined as follows: for any $\mathbf{x}\in [N]^d$, suppose that $\mathbf{x}\in\mathcal{D}_{\mathbf{z},j}$ and $\mathbf{x}=\mathbf{w}_{\mathbf{z},j,l}$ (where $\mathbf{z}\in [N]^{d-1}$ and $l\in [N]$), then we let
\begin{equation*}
    \sigma_j(\mathbf{x})=(\mathbf{z}_1,\cdots,\mathbf{z}_{j-1},\gamma_{\mathbf{z},j}(l),\mathbf{z}_{j},\cdots,\mathbf{z}_{d-1}).
\end{equation*}
It can be checked that $\sigma_j$ follows the uniform distribution on $\mathcal{M}_{\mathbf{b}}\cap\mathcal{A}_j$.

The generalized hit and run algorithm for sampling from $\tilde{\mathbb{P}}_{\beta,d}$ can be described as follows. Each step of the algorithm consists of two sequential parts:
\begin{itemize}
    \item Starting from $\sigma\in S_{N,d}$, for each $\mathbf{x}\in [N]^d$ and $i\in [d]$, sample independently $u_{\mathbf{x},i}$ from the uniform distribution on $[0,e^{2\beta\mathbf{x}_i\sigma(\mathbf{x})_i}]$. Let $b_{\mathbf{x},i}=\frac{\log(u_{\mathbf{x},i})}{2\beta \mathbf{x}_i}$;
    \item Let $\sigma_0=\sigma$, and let $\mathbf{b}$ and $\mathcal{M}_{\mathbf{b}}$ be defined as above. Sequentially for $j=1,\cdots,d$, let $\mathcal{A}_{j}$ be defined as above, and sample $\sigma_j\in S_{N,d}$ from the uniform distribution on $\mathcal{M}_{\mathbf{b}}\cap\mathcal{A}_j$ (with the practical implementation given as above). The chain then moves from $\sigma$ to $\sigma_d$. 
\end{itemize}

The following lemma justifies the correctness of the generalized algorithm.

\begin{lemma}
The stationary distribution of the above generalized algorithm is given by $\tilde{\mathbb{P}}_{\beta,d}$.
\end{lemma}
\begin{proof}
We denote by $K$ the transition kernel of the generalized algorithm, and $K_{\mathbf{b}}$ the transition kernel of the aforementioned Gibbs sampler for sampling from $\pi_{\mathcal{M}_{\mathbf{b}}}$--the uniform distribution on $\mathcal{M}_{\mathbf{b}}$ (here $\mathbf{b}=(b_{\mathbf{x},i})_{\mathbf{x}\in [N]^d,i\in [d]}$ with $b_{\mathbf{x},i}\in\mathbb{R}$ for every $\mathbf{x}\in [N]^d,i\in [d]$). From the theory of Gibbs samplers, the stationary distribution of $K_{\mathbf{b}}$ is given by $\pi_{\mathcal{M}_{\mathbf{b}}}$. Hence for any $\sigma,\sigma'\in S_{N,d}$,
\begin{equation*}
    \sum_{\sigma\in S_{N,d}}1_{\sigma\in\mathcal{M}_{\mathbf{b}}}  K_{\mathbf{b}}(\sigma,\sigma')=1_{\sigma'\in\mathcal{M}_{\mathbf{b}}}.
\end{equation*}

For any $\sigma\in S_{N,d}$, we denote by
\begin{equation*}
 q_{\beta,d}(\sigma):=\exp\Big(2\beta \sum_{\mathbf{x}\in [N]^d}\sum_{i=1}^d \mathbf{x}_i  \sigma(\mathbf{x})_i\Big).
\end{equation*}
For any $\sigma,\sigma'\in S_{N,d}$, we have
\begin{eqnarray*}
  &&\sum_{\sigma\in S_{N,d}}q_{\beta,d}(\sigma)K(\sigma,\sigma') \\
  &=&  \sum_{\sigma \in S_{N,d}}\frac{q_{\beta,d}(\sigma)}{\exp\Big(2\beta  \sum\limits_{\mathbf{x}\in [N]^d}\sum\limits_{i=1}^d\mathbf{x}_i\sigma(\mathbf{x})_i\Big)}\\
  &\times&  \int_{[0,\infty)^{d N^d}}\Big(\prod\limits_{\mathbf{x}\in [N]^d}\prod\limits_{i=1}^d du_{\mathbf{x},i}\Big) 1_{\sigma\in \mathcal{M}_{\mathbf{b}}}1_{\sigma'\in \mathcal{M}_{\mathbf{b}}} K_{\mathbf{b}}(\sigma,\sigma')\\
  &=& \int_{[0,\infty)^{d N^d}}\Big(\prod\limits_{\mathbf{x}\in [N]^d}\prod\limits_{i=1}^d du_{\mathbf{x},i}\Big) 1_{\sigma' \in \mathcal{M}_{\mathbf{b}}}\sum_{\sigma\in S_{N,d}}1_{\sigma\in \mathcal{M}_{\mathbf{b}}} K_{\mathbf{b}}(\sigma,\sigma')\\
  &=& \int_{[0,\infty)^{d N^d}}\Big(\prod\limits_{\mathbf{x}\in [N]^d}\prod\limits_{i=1}^d du_{\mathbf{x},i}\Big)1_{\sigma'\in \mathcal{M}_{\mathbf{b}}}= q_{\beta,d}(\sigma'),
\end{eqnarray*}
where in the third line, $\mathbf{b}$ is obtained from $(u_{\mathbf{x},i})_{\mathbf{x}\in [N]^d,i\in [d]}$ as in the first part of the generalized algorithm. This shows that $\tilde{\mathbb{P}}_{\beta,d}$ is the stationary distribution of the generalized algorithm.
\end{proof}

Similarly, the generalized hit and run algorithm for sampling from $\mathbb{P}_{\beta,d}$ can be described as follows. Each step of the algorithm involves the following two sequential parts:
\begin{itemize}
    \item Starting from $\sigma\in S_{N,d}$, for each $\mathbf{x}\in [N]^d$ and $i\in [d]$, sample independently $u_{\mathbf{x},i}$ from the uniform distribution on $[0,e^{-2\beta (\sigma(\mathbf{x})_i-\mathbf{x}_i)_{+}}]$. Let $b_{\mathbf{x},i}=\mathbf{x}_i-\frac{\log(u_{\mathbf{x},i})}{2\beta}$;
    \item Let $\sigma_0=\sigma$ and $\mathbf{b}$ be defined as above. Let
    \begin{equation*}
        \mathcal{M}'_{\mathbf{b}}=\{\sigma\in S_{N,d}:(\sigma(\mathbf{x}))_i\leq b_{\mathbf{x},i} \text{ for every }\mathbf{x}\in [N]^d,i\in [d] \}. 
    \end{equation*}
    Sequentially for $j=1,\cdots,d$, let $\mathcal{A}_j$ be defined as above, and sample $\sigma_j\in S_{N,d}$ from the uniform distribution on $\mathcal{M}'_{\mathbf{b},j}\cap\mathcal{A}_j$ (with a similar practical implementation as that for $\tilde{P}_{\beta,d}$). The chain then moves from $\sigma$ to $\sigma_d$.
\end{itemize}

At present writing, sharp analysis of these algorithms remains for the future, but extensive simulation studies suggest that they work well.

\section{Simulation studies}\label{Sect.6}

In this section, we display some simulation results to compare the empirical performance of the hit and run algorithm and the Metropolis algorithm for Mallows permutation models with $L^1$ and $L^2$ distances and the two-parameter permutation model with $L^1$ distance and Cayley distance. 

The Metropolis algorithm for sampling from the two-parameter permutation model $\mathbb{P}_{\beta_1,\beta_2}$ is given below (similar to the Metropolis algorithms for the $L^1$ and $L^2$ models introduced in Section \ref{Sect.3.4}). Each step of the algorithm proceeds as follows. Suppose we start at $\sigma$. Independently pick two numbers $i,j\in [n]$ uniformly at random; if $i=j$, stay at $\sigma$; if $i\neq j$, move to $\sigma (i,j)$ with probability 
\begin{eqnarray*}
    &\min\Big\{1,\exp\Big(-\beta_1(|\sigma(i)-j|+|\sigma(j)-i|-|\sigma(i)-i|-|\sigma(j)-j|)\\
    &\quad\quad\quad\quad\quad  
       +\beta_2(C(\sigma(i,j))-C(\sigma))\Big)\Big\},
\end{eqnarray*}
otherwise stay at $\sigma$. It can be checked that the stationary distribution of the algorithm is $\mathbb{P}_{\beta_1,\beta_2}$.

We keep track of the following three statistics associated with a permutation $\sigma\in S_n$: 
\begin{itemize}
    \item The number of fixed points: $T_1(\sigma)=\#\{i\in [n]: \sigma(i)=i\}$;
    \item The length of the cycle that contains $\lceil \frac{n}{2} \rceil$, denoted by $T_2(\sigma)$;
    \item The length of the longest increasing subsequence:
    \begin{equation*}
        T_3(\sigma)=\max\{k\leq n: \sigma(i_1)<\cdots<\sigma(i_k)\text{ for some }1\leq i_1<\cdots<i_k\leq n\}.
    \end{equation*}
\end{itemize}

Now we describe the simulation set-up. Throughout this section we take $n=1000$. For the $L^1$ model, we take $\beta=\frac{1}{n}$; for the $L^2$ model, we take $\beta=\frac{1}{n^2}$; for the two-parameter model, we take $\beta_1=\frac{1}{n}$ and $\beta_2=1$. We note that these choices of $\beta,\beta_1,\beta_2$ are beyond the theoretical analysis in Sections \ref{Sect.3}-\ref{Sect.4}. The simulation results presented in Sections \ref{Sect.6.1}-\ref{Sect.6.2} show that the hit and run algorithm exhibits more desirable performance than the Metropolis algorithm for such choices of parameter values.

To compare the empirical performance of the hit and run algorithm and  the Metropolis algorithm, we list in Table \ref{Table} below the average running times of both algorithms for the three models. 

\begin{table}[!h]
\centering
\begin{tabular}{ |c|c|c|c| }
\hline
 Algorithm & $L^1$ model & $L^2$ model & Two-parameter model \\
 \hline
 Hit and run & $18.1$ & $19.2$ & $76.5$ \\ 
 \hline
 Metropolis & $5.41\cdot 10^{-2}$ & $6.3\cdot 10^{-2}$ & $8.81\cdot 10^{-2}$ \\  
 \hline
\end{tabular}
\caption{Average running time (unit: ms--millisecond) for one step of the hit and run algorithm\slash Metropolis algorithm for the three models}
\label{Table}
\end{table}


\subsection{Trace plots and histograms}\label{Sect.6.1}


In this subsection, we examine the convergence rates of the hit and run algorithm and the Metropolis algorithm through the trace plots and the histograms of the three statistics $\{T_i\}_{i=1}^3$.

First we show some trace plots of $\{T_i\}_{i=1}^3$ for the hit and run algorithm and the Metropolis algorithm started from the identity permutation. The trace plots of the hit and run algorithm are shown in Figure \ref{fig1}, and the trace plots of the Metropolis algorithm are shown in Figure \ref{fig2}. To examine the convergence rates more closely, we also show some histograms of the three statistics $\{T_i\}_{i=1}^3$ at certain steps (started from the identity permutation) for both algorithms in Figures \ref{fig3}-\ref{fig4}, respectively. The histograms for either of the algorithms at step $l$ are obtained by replicating the corresponding chain independently for $4000$ times and collecting the sample at the $l$th step for each replication.

We examine the convergence rates of both algorithms for the $L^1$ model from Subfigures (A1), (A2), (A3), (a1), (a2), and (a3) of Figures \ref{fig1}, \ref{fig2}, \ref{fig3}, and \ref{fig4}. From Subfigures (A1) and (a1) of the four figures, we observe that after running the hit and run algorithm for $2$ steps--which takes $36.2$ milliseconds (ms)--the distribution of $T_1$ is close to stationarity; in comparison, after running the Metropolis algorithm for $4000$ steps--which takes $216.4$ ms--the distribution of $T_1$ is still quite different from stationarity. From Subfigures (A2) and (a2) of the four figures, we learn that after running the hit and run algorithm for $1$ step--which takes $18.1$ ms--the distribution of $T_2$ is close to stationarity, which after running the Metropolis algorithm for $1000$ steps--which takes $54.1$ ms--the distribution of $T_2$ is still far from stationarity. Subfigures (A3) and (a3) of the four figures show that the distribution of $T_3$ is close to stationarity after running the hit and run algorithm for $3$ steps (which takes $54.3$ ms), but is quite different from stationarity after running the Metropolis algorithm for $3000$ steps (which takes $162.3$ ms). In conclusion, in terms of running times, the hit and run algorithm mixes much faster than the Metropolis algorithm for the $L^1$ model. Similar conclusions can be made for the $L^2$ model and the two-parameter permutation model from Figures \ref{fig1}, \ref{fig2}, \ref{fig3}, and \ref{fig4}. 

In Figure \ref{fig4.1}, we also compare the histograms of $\{T_i\}_{i=1}^3$ for the hit and run algorithm (after running the chain for $100$ steps started from the identity permutation for all the three models) and the Metropolis algorithm (after running the chain for $10000$ steps--for the $L^1$ and $L^2$ models--or $30000$ steps--for the two-parameter permutation model--started from the identity permutation); as can be seen from the figure, the histograms for the two algorithms match closely.

\begin{figure}[H]
  \centering
  \includegraphics[
  height=1.4\textwidth,
  width=\textwidth
  ]{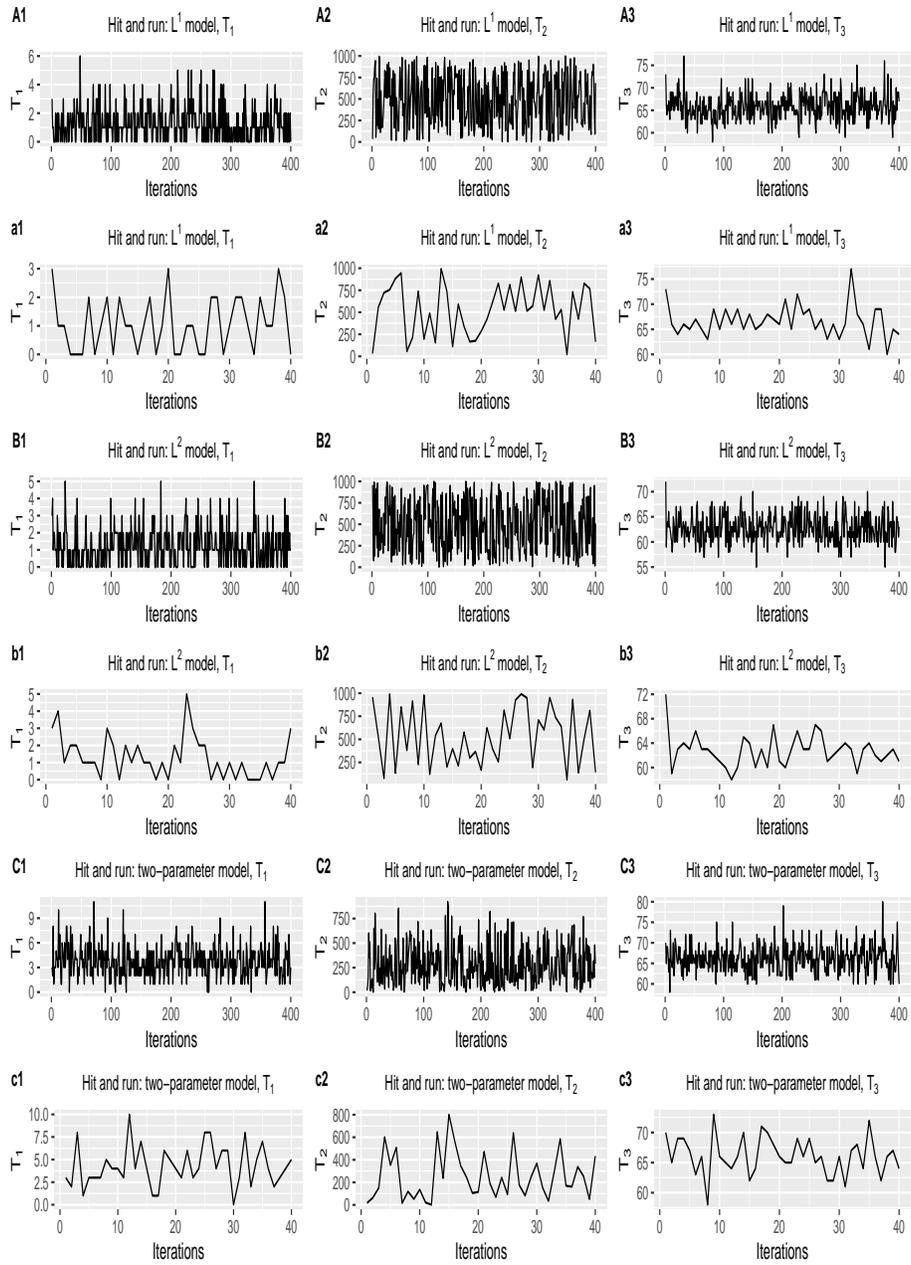}
 \caption{Trace plots of $\{T_i\}_{i=1}^3$ for the hit and run algorithm}
 \label{fig1}
\end{figure} 

\begin{figure}[H]
  \centering
  \includegraphics[
  height=1.4\textwidth,
  width=\textwidth
  ]{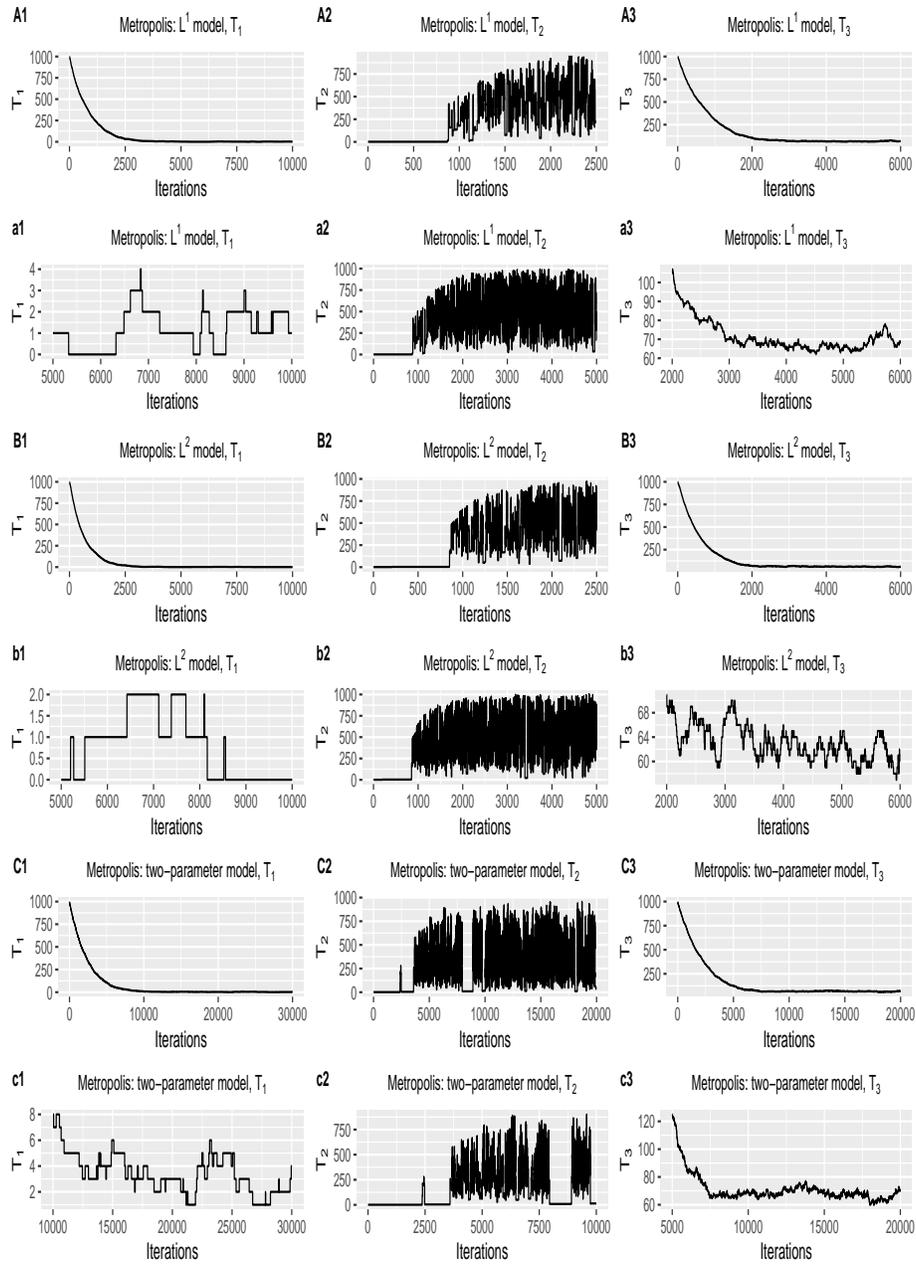}
 \caption{Trace plots of $\{T_i\}_{i=1}^3$ for the Metropolis algorithm}
 \label{fig2}
\end{figure} 

\begin{figure}[H]
  \centering
  \includegraphics[
  height=1.4\textwidth,
  width=\textwidth
  ]{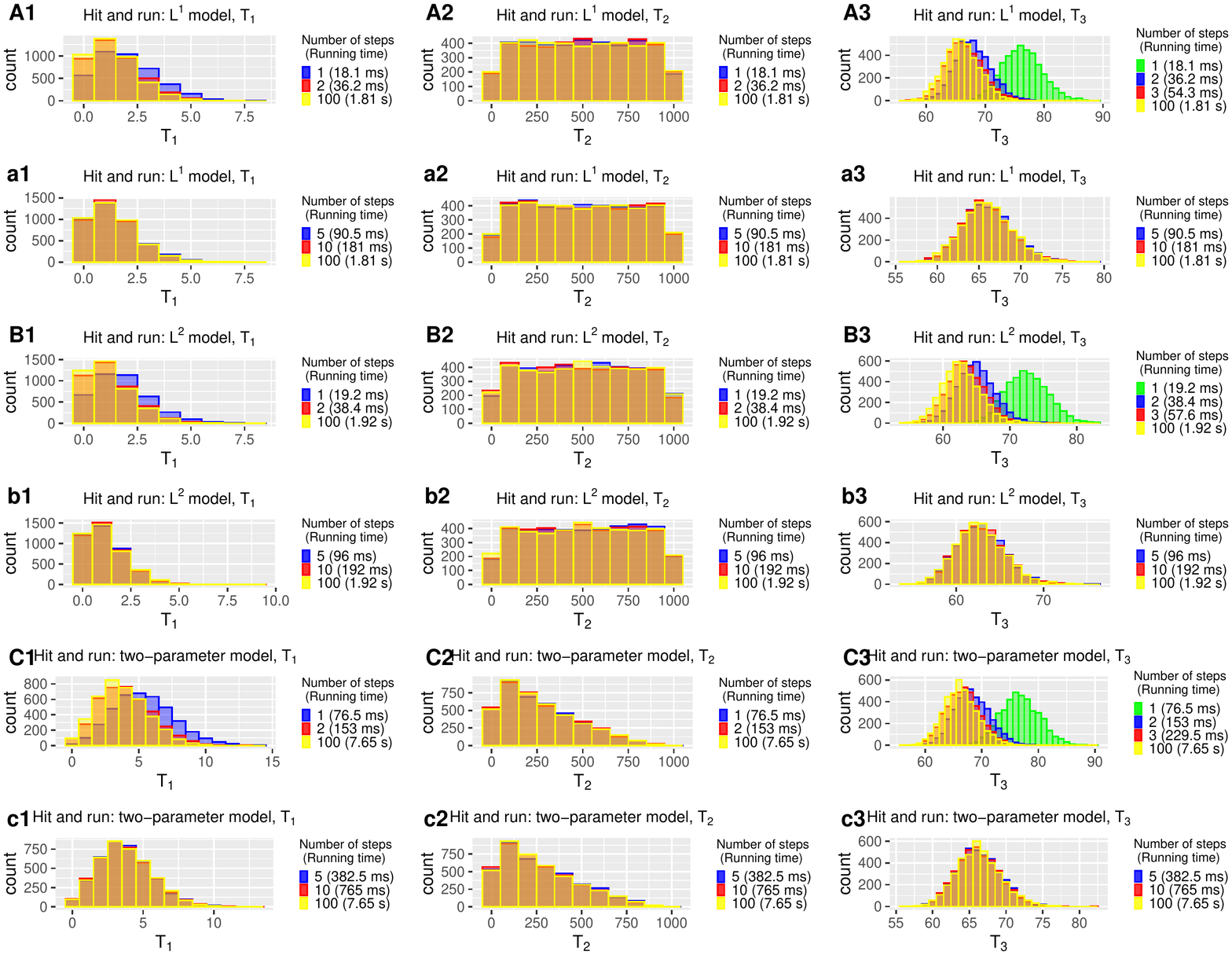}
 \caption{Histograms of $\{T_i\}_{i=1}^3$ for the hit and run algorithm: running time is computed as the product of the number of steps and the corresponding average running time as listed in Table \ref{Table}}
 \label{fig3}
\end{figure} 

\begin{figure}[H]
  \centering
  \includegraphics[
  height=1.4\textwidth,
  width=\textwidth
  ]{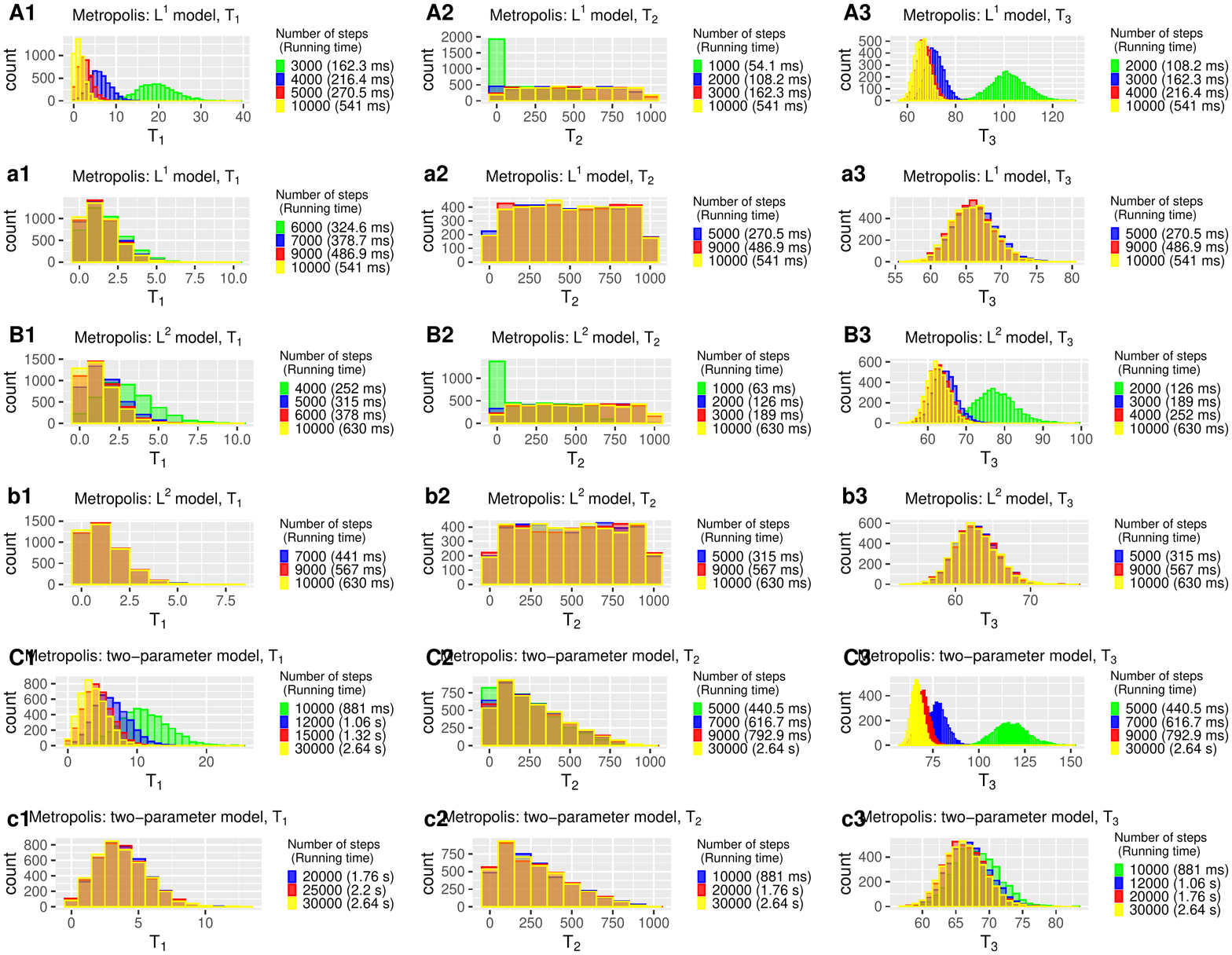}
 \caption{Histograms of $\{T_i\}_{i=1}^3$ for the Metropolis algorithm: running time is computed as the product of the number of steps and the corresponding average running time as listed in Table \ref{Table}}
 \label{fig4}
\end{figure} 

\begin{figure}[H]
  \centering
  \includegraphics[
  height=\textwidth,
  width=\textwidth
  ]{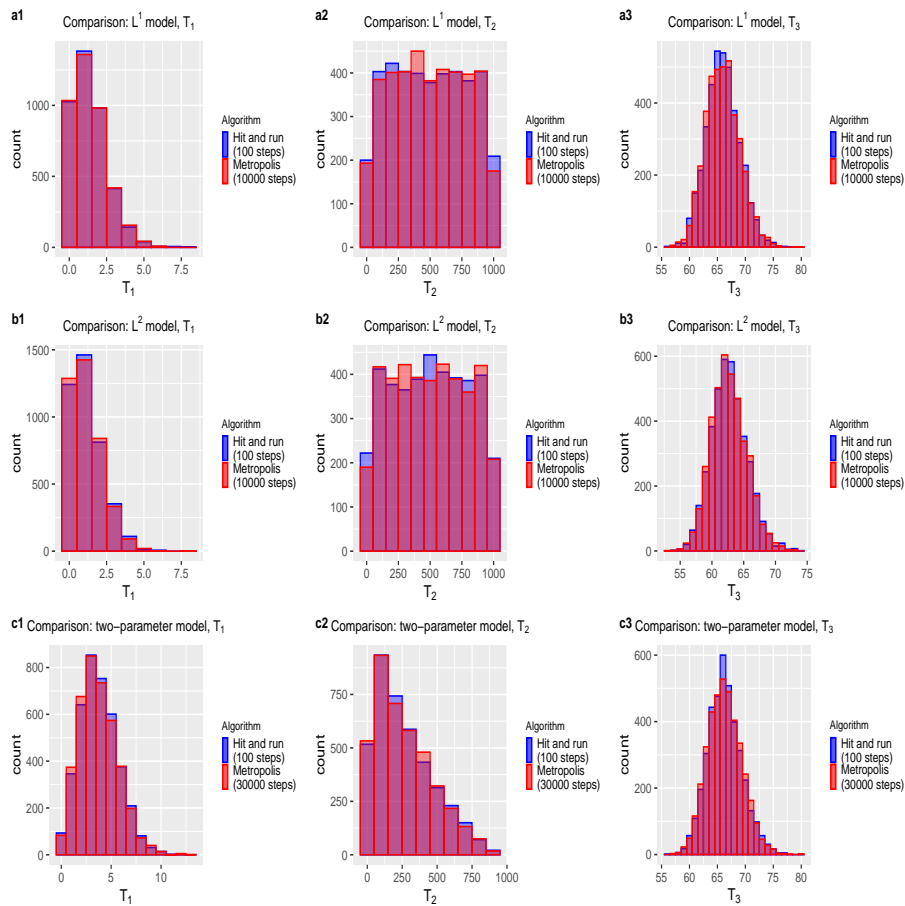}
 \caption{Comparison of the histograms of $\{T_i\}_{i=1}^3$ for the hit and run algorithm and the Metropolis algorithm}
 \label{fig4.1}
\end{figure}

\subsection{Autocorrelations}\label{Sect.6.2}

The trace plots in Section \ref{Sect.6.1} indicate that there is a stronger serial correlation of the Metropolis algorithm than that of the hit and run algorithm. To investigate this aspect more closely, we plot the autocorrelations of $\{T_i\}_{i=1}^3$ for both algorithms in Figure \ref{fig5}. For all the three models, the autocorrelation plot for the hit and run algorithm is based on $1000$ samples started from the identity permutation. The autocorrelation plot for the Metropolis algorithm is based on $10000$ samples (for the $L^1$ and $L^2$ models) or $30000$ samples (for the two-parameter permutation model) started from the identity permutation. 

From the autocorrelation plots, we observe that for all the three statistics $\{T_i\}_{i=1}^3$, the autocorrelation for the Metropolis algorithm is much larger than that of the hit and run algorithm. Moreover, for the Metropolis algorithm, the autocorrelation of $T_2$ is considerably smaller than that of $T_1$ and $T_3$. The former observation suggests that if we collect the same number of samples from both algorithms, the effective sample size for the hit and run algorithm can be much larger than that for the Metropolis algorithm. Hence in terms of effective sample size, the hit and run algorithm is also much more efficient than the Metropolis algorithm.

\begin{figure}[H]
  \centering
  \includegraphics[
  height=1.4\textwidth,
  width=\textwidth
  ]{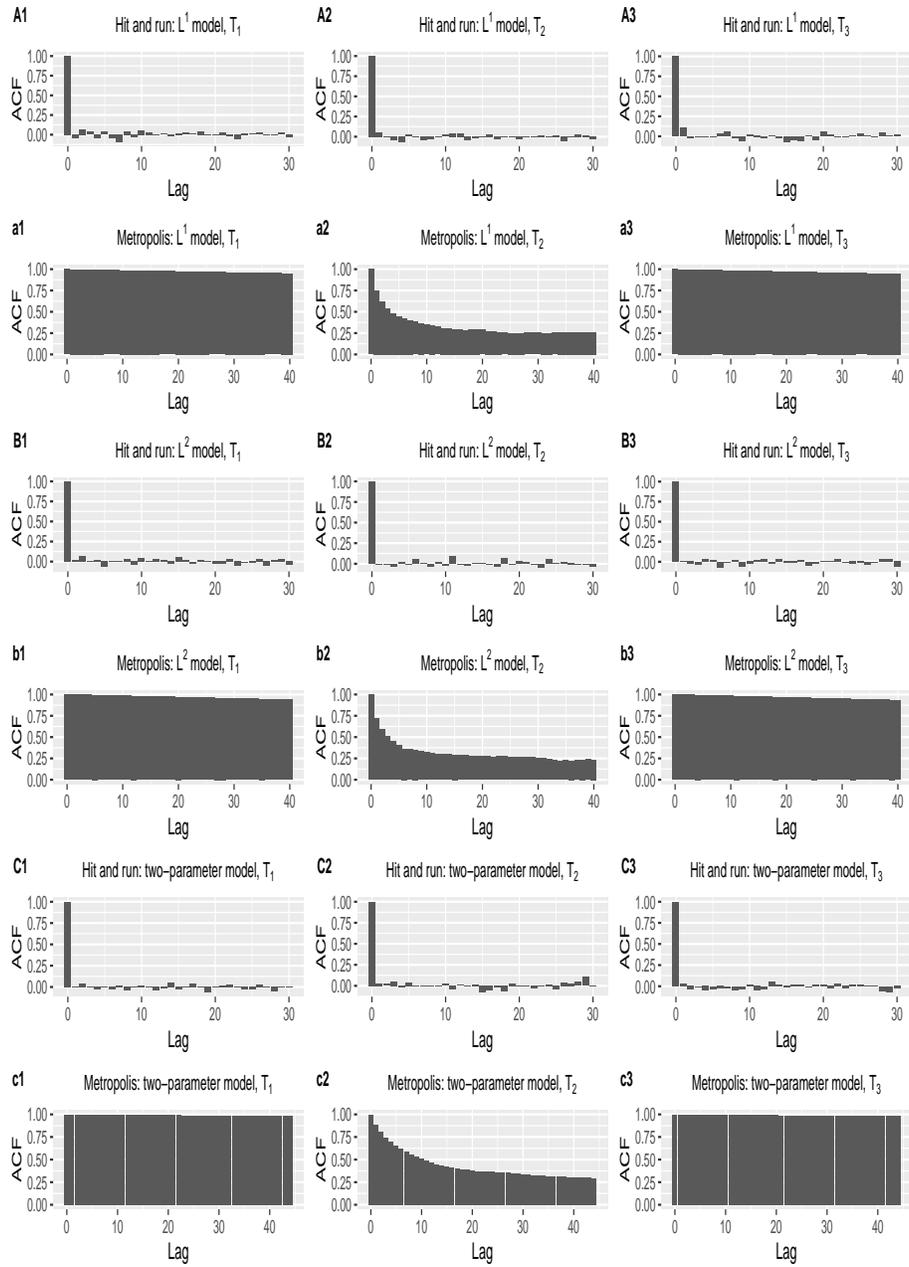}
 \caption{Autocorrelation plots of $\{T_i\}_{i=1}^3$ for the hit and run algorithm and the Metropolis algorithm}
 \label{fig5}
\end{figure} 

\newpage

\bibliographystyle{acm}
\bibliography{Mallows.bib}

\begin{thebibliography}{10}

\bibitem{AC}
{\sc Addario-Berry, L., and Corsini, B.}
\newblock The height of {M}allows trees.
\newblock {\em Ann. Probab. 49}, 5 (2021), 2220--2271.

\bibitem{Aldous}
{\sc Aldous, D.~J.}
\newblock Exchangeability and related topics.
\newblock In {\em \'{E}cole d'\'{e}t\'{e} de probabilit\'{e}s de
  {S}aint-{F}lour, {XIII}---1983}, vol.~1117 of {\em Lecture Notes in Math.}
  Springer, Berlin, 1985, pp.~1--198.

\bibitem{AD}
{\sc Andersen, H.~C., and Diaconis, P.}
\newblock Hit and run as a unifying device.
\newblock {\em J. Soc. Fr. Stat. \& Rev. Stat. Appl. 148}, 4 (2007), 5--28.

\bibitem{ABT}
{\sc Arratia, R., Barbour, A.~D., and Tavar\'{e}, S.}
\newblock {\em Logarithmic combinatorial structures: a probabilistic approach}.
\newblock EMS Monographs in Mathematics. European Mathematical Society (EMS),
  Z\"{u}rich, 2003.

\bibitem{ABSV}
{\sc Awasthi, P., Blum, A., Sheffet, O., and Vijayaraghavan, A.}
\newblock Learning mixtures of ranking models.
\newblock {\em arXiv preprint arXiv:1410.8750\/} (2014).

\bibitem{BB}
{\sc Basu, R., and Bhatnagar, N.}
\newblock Limit theorems for longest monotone subsequences in random {M}allows
  permutations.
\newblock {\em Ann. Inst. Henri Poincar\'{e} Probab. Stat. 53}, 4 (2017),
  1934--1951.

\bibitem{Bet}
{\sc Betz, V.}
\newblock Random permutations of a regular lattice.
\newblock {\em J. Stat. Phys. 155}, 6 (2014), 1222--1248.

\bibitem{BP}
{\sc Bhatnagar, N., and Peled, R.}
\newblock Lengths of monotone subsequences in a {M}allows permutation.
\newblock {\em Probab. Theory Related Fields 161}, 3-4 (2015), 719--780.

\bibitem{BR}
{\sc Biskup, M., and Richthammer, T.}
\newblock Gibbs measures on permutations over one-dimensional discrete point
  sets.
\newblock {\em Ann. Appl. Probab. 25}, 2 (2015), 898--929.

\bibitem{Blum}
{\sc Blumberg, O.}
\newblock Cutoff for the transposition walk on permutations with one-sided
  restrictions.
\newblock {\em arXiv preprint arXiv:1202.4797\/} (2012).

\bibitem{BRS}
{\sc Boardman, S., Rudolf, D., and Saloff-Coste, L.}
\newblock The hit-and-run version of top-to-random.
\newblock {\em arXiv preprint arXiv:2009.04977\/} (2020).

\bibitem{BDJ2}
{\sc Borodin, A., Diaconis, P., and Fulman, J.}
\newblock On adding a list of numbers (and other one-dependent determinantal
  processes).
\newblock {\em Bull. Amer. Math. Soc. (N.S.) 47}, 4 (2010), 639--670.

\bibitem{BD}
{\sc Bubley, R., and Dyer, M.}
\newblock Path coupling: A technique for proving rapid mixing in {Markov}
  chains.
\newblock In {\em Proceedings 38th Annual Symposium on Foundations of Computer
  Science\/} (1997), IEEE, pp.~223--231.

\bibitem{CBBK}
{\sc Chen, H., Branavan, S., Barzilay, R., and Karger, D.~R.}
\newblock Content modeling using latent permutations.
\newblock {\em Journal of Artificial Intelligence Research 36\/} (2009),
  129--163.

\bibitem{Crane}
{\sc Crane, H.}
\newblock The ubiquitous {E}wens sampling formula.
\newblock {\em Statist. Sci. 31}, 1 (2016), 1--19.

\bibitem{CDE}
{\sc Crane, H., DeSalvo, S., and Elizalde, S.}
\newblock The probability of avoiding consecutive patterns in the {M}allows
  distribution.
\newblock {\em Random Structures Algorithms 53}, 3 (2018), 417--447.

\bibitem{Cri}
{\sc Critchlow, D.~E.}
\newblock {\em Metric methods for analyzing partially ranked data}, vol.~34 of
  {\em Lecture Notes in Statistics}.
\newblock Springer-Verlag, Berlin, 1985.

\bibitem{CFV}
{\sc Critchlow, D.~E., Fligner, M.~A., and Verducci, J.~S.}
\newblock Probability models on rankings.
\newblock {\em J. Math. Psych. 35}, 3 (1991), 294--318.

\bibitem{D}
{\sc Diaconis, P.}
\newblock {\em Group representations in probability and statistics}, vol.~11 of
  {\em Institute of Mathematical Statistics Lecture Notes---Monograph Series}.
\newblock Institute of Mathematical Statistics, Hayward, CA, 1988.

\bibitem{Dia}
{\sc Diaconis, P.}
\newblock Analysis of a {B}ose-{E}instein {M}arkov chain.
\newblock {\em Ann. Inst. H. Poincar\'{e} Probab. Statist. 41}, 3 (2005),
  409--418.

\bibitem{DGH}
{\sc Diaconis, P., Graham, R., and Holmes, S.~P.}
\newblock Statistical problems involving permutations with restricted
  positions.
\newblock In {\em State of the art in probability and statistics ({L}eiden,
  1999)}, vol.~36 of {\em IMS Lecture Notes Monogr. Ser.} Inst. Math. Statist.,
  Beachwood, OH, 2001, pp.~195--222.

\bibitem{DH}
{\sc Diaconis, P., and Hanlon, P.}
\newblock Eigen-analysis for some examples of the {M}etropolis algorithm.
\newblock In {\em Hypergeometric functions on domains of positivity, {J}ack
  polynomials, and applications ({T}ampa, {FL}, 1991)}, vol.~138 of {\em
  Contemp. Math.} Amer. Math. Soc., Providence, RI, 1992, pp.~99--117.

\bibitem{DR}
{\sc Diaconis, P., and Ram, A.}
\newblock Analysis of systematic scan {M}etropolis algorithms using
  {I}wahori-{H}ecke algebra techniques.
\newblock vol.~48. 2000, pp.~157--190.
\newblock Dedicated to William Fulton on the occasion of his 60th birthday.

\bibitem{DR3}
{\sc Diaconis, P., and Ram, A.}
\newblock A probabilistic interpretation of the {M}acdonald polynomials.
\newblock {\em Ann. Probab. 40}, 5 (2012), 1861--1896.

\bibitem{DS}
{\sc Diaconis, P., and Simper, M.}
\newblock Statistical enumeration of groups by double cosets.
\newblock {\em Journal of Algebra\/} (2021).

\bibitem{DZ}
{\sc Diaconis, P., and Zhong, C.}
\newblock Hahn polynomials and the {Burnside} process.
\newblock {\em The Ramanujan Journal\/} (2021), 1--29.

\bibitem{FC}
{\sc Feigin, P.~D., and Cohen, A.}
\newblock On a model for concordance between judges.
\newblock {\em Journal of the Royal Statistical Society: Series B
  (Methodological) 40}, 2 (1978), 203--213.

\bibitem{Feray}
{\sc F\'{e}ray, V.}
\newblock Asymptotic behavior of some statistics in {E}wens random
  permutations.
\newblock {\em Electron. J. Probab. 18\/} (2013), no. 76, 32.

\bibitem{Fey}
{\sc Feynman, R.~P.}
\newblock Atomic theory of the $\lambda$ transition in helium.
\newblock {\em Physical Review 91}, 6 (1953), 1291.

\bibitem{Fic}
{\sc Fichtner, K.-H.}
\newblock Random permutations of countable sets.
\newblock {\em Probab. Theory Related Fields 89}, 1 (1991), 35--60.

\bibitem{FV}
{\sc Fligner, M.~A., and Verducci, J.~S.}
\newblock Distance based ranking models.
\newblock {\em J. Roy. Statist. Soc. Ser. B 48}, 3 (1986), 359--369.

\bibitem{FV2}
{\sc Fligner, M.~A., and Verducci, J.~S.}
\newblock Multistage ranking models.
\newblock {\em J. Amer. Statist. Assoc. 83}, 403 (1988), 892--901.

\bibitem{FM}
{\sc Fyodorov, Y.~V., and Muirhead, S.}
\newblock The band structure of a model of spatial random permutation.
\newblock {\em Probab. Theory Related Fields 179}, 3-4 (2021), 543--587.

\bibitem{GRU}
{\sc Gandolfo, D., Ruiz, J., and Ueltschi, D.}
\newblock On a model of random cycles.
\newblock {\em J. Stat. Phys. 129}, 4 (2007), 663--676.

\bibitem{GP}
{\sc Gladkich, A., and Peled, R.}
\newblock On the cycle structure of {M}allows permutations.
\newblock {\em Ann. Probab. 46}, 2 (2018), 1114--1169.

\bibitem{GO}
{\sc Gnedin, A., and Olshanski, G.}
\newblock The two-sided infinite extension of the {M}allows model for random
  permutations.
\newblock {\em Adv. in Appl. Math. 48}, 5 (2012), 615--639.

\bibitem{GLU}
{\sc Grosskinsky, S., Lovisolo, A.~A., and Ueltschi, D.}
\newblock Lattice permutations and {P}oisson-{D}irichlet distribution of cycle
  lengths.
\newblock {\em J. Stat. Phys. 146}, 6 (2012), 1105--1121.

\bibitem{He}
{\sc He, J.}
\newblock A central limit theorem for descents of a {Mallows} permutation and
  its inverse.
\newblock {\em arXiv preprint arXiv:2005.09802\/} (2020).

\bibitem{HHL}
{\sc Holroyd, A.~E., Hutchcroft, T., and Levy, A.}
\newblock Mallows permutations and finite dependence.
\newblock {\em Ann. Probab. 48}, 1 (2020), 343--379.

\bibitem{Ka3}
{\sc Kammoun, M.~S.}
\newblock Monotonous subsequences and the descent process of invariant random
  permutations.
\newblock {\em Electron. J. Probab. 23\/} (2018), Paper no. 118, 31.

\bibitem{Ka1}
{\sc Kammoun, M.~S.}
\newblock On the longest common subsequence of conjugation invariant random
  permutations.
\newblock {\em Electron. J. Combin. 27}, 4 (2020), Paper No. 4.10, 21.

\bibitem{Ka2}
{\sc Kammoun, M.~S.}
\newblock Universality for random permutations and some other groups.
\newblock {\em arXiv preprint arXiv:2012.05845\/} (2020).

\bibitem{LL}
{\sc Lebanon, G., and Lafferty, J.}
\newblock Cranking: {Combining} rankings using conditional probability models
  on permutations.
\newblock In {\em ICML\/} (2002), vol.~2, pp.~363--370.

\bibitem{LM2}
{\sc Lebanon, G., and Mao, Y.}
\newblock Non-parametric modeling of partially ranked data.
\newblock {\em Journal of Machine Learning Research 9}, 10 (2008).

\bibitem{LP}
{\sc Levin, D.~A., and Peres, Y.}
\newblock {\em Markov chains and mixing times}.
\newblock American Mathematical Society, Providence, RI, 2017.
\newblock Second edition of [ MR2466937], With contributions by Elizabeth L.
  Wilmer, With a chapter on ``Coupling from the past'' by James G. Propp and
  David B. Wilson.

\bibitem{Lovasz1}
{\sc Lov\'{a}sz, L.}
\newblock Hit-and-run mixes fast.
\newblock {\em Math. Program. 86}, 3, Ser. A (1999), 443--461.

\bibitem{LV1}
{\sc Lov{\'a}sz, L., and Vempala, S.}
\newblock Hit-and-run is fast and fun.
\newblock {\em preprint, Microsoft Research\/} (2003).

\bibitem{LV2}
{\sc Lov\'{a}sz, L., and Vempala, S.}
\newblock Hit-and-run from a corner.
\newblock In {\em Proceedings of the 36th {A}nnual {ACM} {S}ymposium on
  {T}heory of {C}omputing\/} (2004), ACM, New York, pp.~310--314.

\bibitem{Mal}
{\sc Mallows, C.~L.}
\newblock Non-null ranking models. {I}.
\newblock {\em Biometrika 44\/} (1957), 114--130.

\bibitem{Mar}
{\sc Marden, J.~I.}
\newblock {\em Analyzing and modeling rank data}, vol.~64 of {\em Monographs on
  Statistics and Applied Probability}.
\newblock Chapman \& Hall, London, 1995.

\bibitem{MPPB}
{\sc Meila, M., Phadnis, K., Patterson, A., and Bilmes, J.~A.}
\newblock Consensus ranking under the exponential model.
\newblock {\em arXiv preprint arXiv:1206.5265\/} (2012).

\bibitem{MS}
{\sc Mueller, C., and Starr, S.}
\newblock The length of the longest increasing subsequence of a random
  {M}allows permutation.
\newblock {\em J. Theoret. Probab. 26}, 2 (2013), 514--540.

\bibitem{M1}
{\sc Mukherjee, S.}
\newblock Estimation in exponential families on permutations.
\newblock {\em Ann. Statist. 44}, 2 (2016), 853--875.

\bibitem{M2}
{\sc Mukherjee, S.}
\newblock Fixed points and cycle structure of random permutations.
\newblock {\em Electron. J. Probab. 21\/} (2016), Paper No. 40, 18.

\bibitem{Pins}
{\sc Pinsky, R.~G.}
\newblock Permutations avoiding a pattern of length three under {M}allows
  distributions.
\newblock {\em Random Structures Algorithms 58}, 4 (2021), 676--690.

\bibitem{PT}
{\sc Pitman, J., and Tang, W.}
\newblock Regenerative random permutations of integers.
\newblock {\em Ann. Probab. 47}, 3 (2019), 1378--1416.

\bibitem{Sta}
{\sc Stanley, R.~P.}
\newblock {\em Enumerative combinatorics. {V}olume 1}, second~ed., vol.~49 of
  {\em Cambridge Studies in Advanced Mathematics}.
\newblock Cambridge University Press, Cambridge, 2012.

\bibitem{Starr}
{\sc Starr, S.}
\newblock Thermodynamic limit for the {M}allows model on {$S_n$}.
\newblock {\em J. Math. Phys. 50}, 9 (2009), 095208, 15.

\bibitem{PS}
{\sc Switzer, P.}
\newblock Personal communication.

\bibitem{Toth}
{\sc T\'{o}th, B.}
\newblock Improved lower bound on the thermodynamic pressure of the spin
  {$1/2$} {H}eisenberg ferromagnet.
\newblock {\em Lett. Math. Phys. 28}, 1 (1993), 75--84.

\bibitem{Turcin}
{\sc Tur\v{c}in, V.~F.}
\newblock On the computation of multidimensional integrals according to the
  {M}onte {C}arlo method.
\newblock {\em Teor. Verojatnost. i Primenen. 16\/} (1971), 738--743.

\bibitem{Zho2}
{\sc Zhong, C.}
\newblock Mallows permutation models with {$L^1$} and {$L^2$} distances
  \rom{2}: band structure and limiting profile, forthcoming.

\bibitem{Zho3}
{\sc Zhong, C.}
\newblock Mallows permutation models with {$L^1$} and {$L^2$} distances
  \rom{3}: longest increasing subsequence, forthcoming.

\bibitem{Zho4}
{\sc Zhong, C.}
\newblock Mallows permutation models with {$L^1$} and {$L^2$} distances
  \rom{4}: fixed points and cycle structure, forthcoming.

\end{thebibliography}
\end{document}